\documentclass[12pt, reqno]{amsart}

\usepackage{graphicx}
\usepackage[a4paper,vmargin={3cm,3cm},hmargin={3cm,3cm},bottom=30mm]{geometry}
\usepackage{amsmath,amsfonts,amscd,amssymb,amsthm}
\usepackage{anyfontsize}
\usepackage{subcaption}
\usepackage{indentfirst}
\numberwithin{equation}{section}
\usepackage{todonotes}

\usepackage[bookmarks]{hyperref}

\newtheorem{theorem}{Theorem}[section]
\newtheorem{definition}[theorem]{Definition}

\newtheorem{lemma}[theorem]{Lemma}
\newtheorem{proposition}[theorem]{Proposition}

\newtheorem{remark}{Remark}

\numberwithin{equation}{section}

\usepackage{color}
\definecolor{Red}{cmyk}{0,1,1,0}

\definecolor{Blue}{cmyk}{1,1,0,0}

\definecolor{Green}{cmyk}{1,0,1,0}

\DeclareMathOperator{\dis}{dis}

\DeclareMathOperator{\ext}{Ext}
\DeclareMathOperator{\err}{Err}
\DeclareMathOperator{\ahom}{\bar{\mathbf{a}}}
\DeclareMathOperator{\nan}{\nabla \cdot \mathbf{a} \nabla}
\DeclareMathOperator{\nad}{\nabla\cdot \mathbf{a} \nabla}
\DeclareMathOperator{\naN}{\nabla_N \cdot \mathbf{a}_N \nabla_N}
\DeclareMathOperator{\naeps}{\nabla \cdot \mathbf{a}_\varepsilon \nabla}
\DeclareMathOperator{\nah}{\nabla \cdot \ahom \nabla}
\DeclareMathOperator{\sym}{sym}

\DeclareMathOperator{\1}{\mathbf{1}}
\DeclareMathOperator{\Id}{Id}

\newcommand{\mc}[1]{{\mathcal #1}}

\newcommand{\bb}[1]{{\mathbb #1}}

\newcommand{\<}{\langle}
\renewcommand{\>}{\rangle}
\renewcommand{\>}{\rangle}

\def\1{{\mathchoice {\rm 1\mskip-4mu l} {\rm 1\mskip-4mu l}
{\rm 1\mskip-4.5mu l} {\rm 1\mskip-5mu l}}}

\title[Stochastic homogenization of Gaussian fields on random media]{Stochastic homogenization of Gaussian fields on random media}
\author{Leandro Chiarini} 
\address[Utrecht University]{Mathematical Institute}
\address{Budapestlaan 6, 3584 CD Utrecht, The Netherlands}
\email{l.chiarinimedeiros@uu.nl}          

\author{Wioletta M. Ruszel} 
\email{w.m.ruszel@uu.nl}          

\begin{document}


\maketitle

\begin{abstract}
In this article, we study stochastic homogenization of non-homogeneous Gaussian free fields $\Xi^{g,{\bf a}} $ and bi-Laplacian fields $\Xi^{b,{\bf a}}$.
They can be characterized as follows: for $f=\delta$ the solution $u$ of $\nabla \cdot \mathbf{a} \nabla u =f$, ${\bf a}$ is a uniformly elliptic random environment, is the covariance of $\Xi^{g,{\bf a}}$.
When $f$ is the white noise, the field $\Xi^{b,{\bf a}}$ can be viewed as the distributional solution of the same elliptic equation.
Our results characterize the scaling limit of such fields on both, a sufficiently regular domain $D\subset \mathbb{R}^d$, or on the discrete torus.
Based on stochastic homogenization techniques applied to the eigenfunction basis of the Laplace operator $\Delta$, we will show that such families of fields converge to an appropriate multiple of the GFF resp. bi-Laplacian. The limiting fields are determined by their respective homogenized operator $\ahom \Delta$, with constant $\ahom$ depending on the law of the environment ${\bf a}$.
The proofs are based on the results found in \cite{Armstrong2019} and \cite{gloria2014optimal}.
\end{abstract}

\section{Introduction}\label{sec:intro} 

Heat flow through materials with randomly placed impurities (porous media) are typically modelled using parabolic or elliptic equations in divergence form and diffusion matrices ${\bf a}$ modelling the environment, \cite{ap1,ap2,ap3}. 

The steady-state solution $u$ of some corresponding parabolic PDE satisfies
\begin{equation}\label{eq:sss} 
	- \nabla \cdot {\bf a}(x) \nabla u(x)  = f(x) 
\end{equation}
on $D$ and $u=0$ on $\partial D$, where ${\bf a}: \bb R^d \to \bb R^{d \times d}$ is a matrix valued function of the random environment. Provided that certain hypotheses over ${\bf a}$ are true, e.g. that ${\bf a}$ is stationary, there exists a matrix $\ahom$ such that the solution $u$ of the equation above can be well approximated by the solution $\bar{u}$ with constant coefficients $\ahom$ given by
\begin{equation}\label{eq:sss2} 
	- \nah \bar{u}(x)  = f(x) 
\end{equation}
on $D$ and $\bar{u}=0$ on $\partial D$. 
In all cases that we consider in this article, $\ahom = c \Id$, where $\Id$ is the $d\times d$ identity matrix and $c \in \bb R_{+}$ is deterministic constant that only depends on the law of the environment. We abuse notation by denoting the constant $c$ by the same symbol $\ahom$ to simplify notation.

Stochastic homogenization aims at identifying $\ahom$  and quantifying the convergence of the solution $u$ of the heterogeneous equation to the solution of an appropriate deterministic constant-coefficient equation $\bar{u}$.

To define the approach formally, we fix the domain $D$ and introduce a parameter $0 < \varepsilon \ll 1$ that represents the ratio of microscopic and macroscopic scales. The physical intuition behind this approach is the assumption that the coefficients modelling the porous medium vary on a microscopic scale whereas on a macroscale the random environment shows effective behavior.

Then, we rescale the coefficient field ${\bf a}$ in equation \eqref{eq:sss} by defining
\begin{equation}\label{eq:1}
  -\nabla \cdot {\bf a}_{\varepsilon} \nabla u^{\varepsilon} = f,
\end{equation}
where ${\bf a}_{\varepsilon}(\cdot) = {\bf a}(\cdot/\varepsilon)$.

One can then prove that in fact, under certain the conditions on the law of the field ${\bf a}$, $u^{\varepsilon}$ converges to $\bar{u}$ as $\varepsilon \to 0$ in some appropriate functional space. 

There is a rich mathematical theory developed on stochastic homogenisation, \cite{Armstrong2019, tar, komo, braides, jurinskii1980dirichlet, kozlov1979averaging, papanicolau1979boundary, gloria2014optimal, Bella2017StochasticHO} just to mention a few.
In particular, in the last few years, quantitative results about the convergences as $\varepsilon \to 0$ were developed.

In this article, we want to use tools of quantitative stochastic homogenization in order to understand large-scale behaviour of  the non-homogeneous Gaussian free field (GFF) and bi-Laplacian field, whose covariance structure is related to non-homogeneous differential operator $\nad$, in a domain $D\subset \mathbb{R}^d$ and on the discrete torus. Those fields appear as distributional solutions of the equation \eqref{eq:sss} for different choices of $f$, i.e. $f=\delta$ Dirac delta resp. $f=\xi$ white noise.

Their homogeneous counterparts, that is, the (homogeneous) GFF $\Xi^g_D$ and bi-Laplacian field $\Xi^b_D$, are the two  most prominent examples of the family of fractional Gaussian fields, \cite{Lodhia2016}.
The GFF has a wide range of connections to e.g. scaling limits of observables of interacting particle systems \cite{Kenyon}, quantum field theory \cite{Berestycki2016IntroductionTT}, and invariant measures of Allan-Cahn type stochastic partial differential equations \cite{Hairer}.
The bi-Laplacian field or membrane model is related to scaling limits of uniform spanning trees \cite{Sun2013UniformSF}, odometer functions of sandpile models \cite{Levine, cipriani2018scaling} and general interface models \cite{funaki}.

Our main result, Theorem \ref{thm:ConvCont} states that the non-homogeneous GFF and bi-Laplacian field (see Definition~\ref{def:GFFNHom} and \ref{def:BiLapNHom}) converge
\begin{equation}\label{eq:3}
  \Xi^{g,{\bf a}_{\varepsilon}}_D
  \longrightarrow \ahom^{-1/2}\Xi^g_D 
\end{equation} 
and
\begin{equation}\label{eq:4} 
 \Xi^{b, {\bf a}_{\varepsilon}}_D
  \longrightarrow \ahom^{-1}\Xi^b_D
\end{equation} 
as $\varepsilon \to 0$,  respectively in some appropriate Sobolev space.
In the second result, Theorem~\ref{thm:ConvDisc}, we will prove the equivalent of \eqref{eq:3} and \eqref{eq:4} starting in a discretized torus.

To the best of the authors knowledge, the results presented in this work are novel and have not been previously published in the literature. 
In particular, our work extends the techniques outlined in previous papers such as \cite{Armstrong2019,gloria2014optimal} for the stochastic homogenization analysis of the bi-Laplacian.
This extension involves not only accounting for the randomness resulting from $f=\xi$, but also accommodating the lack of a point-wise definition of the white-noise by viewing it to be a fixed distribution with a predetermined negative Sobolev regularity.

The reader familiar with the GFF and bi-Laplacian field will notice that we prove convergence in a  Sobolev space with regularity strictly weaker than the optimal regularity of the GFF and the bi-Laplacian field, due to the existence of the random environment.
In fact, we do not expect that convergence of these fields would hold in such topologies, as analogue effects are also found in \cite{gloria2011optimal,Armstrong2019}. 

The main novelty of our proof is a convenient expansion of $G^\varepsilon$, the  Green's function of the operator $\naeps$,  in terms of  $\{\phi_k\}_{k\in \bb N}$, the orthonormal eigenfunction  basis of $\Delta$.
Using the fact that the operator $\naeps$ in \eqref{eq:1} is self-adjoint, we have that
\begin{equation*}
      G^\varepsilon(x,y)
      = \sum_{k = 1}^{\infty} 
      \frac{\phi_k(x) \varphi^\varepsilon_k(y)}{\ahom\lambda_k}
      = \sum_{k = 1 }^{\infty}
      \frac{\varphi^\varepsilon_k(x) \phi_k(y)}{\ahom\lambda_k},
\end{equation*}
where $\naeps \varphi^\varepsilon_k = -\lambda_k \ahom \phi_k$, $\lambda_k$ is the eigenvalue of $\Delta$ associated with $\phi_k$. With this expansion, we are able to push results from quantitative stochastic homogenization of continuous functions to convergence of the Green's function, resulting in the convergence of the non-homogeneous random fields.

The manuscript is organized as follows.
In Section \ref{sec:not}, we will introduce the notation and define  non-homogeneous Gaussian free fields and bi-Laplacian fields in regular domains and on the torus.
Our main results are presented in Section~\ref{sec:results}, and their proofs are postponed to Section~\ref{sec:proofs}.
Finally, in Section~\ref{sec:conclusion}, we discuss some possible extensions and cite related results regarding large scale behaviour of Gaussian fields in random environments.

\section{Notation and preliminaries}\label{sec:not}

\subsection{Function spaces}\label{sec:spaces}
\subsubsection{Continuous spaces}

First, we need to define the types of domains for which our results hold
\begin{definition}[Regular domains]\label{def:C11-domain}
Let $d \ge 2$ and $D \subset \bb R^d$ be a domain.
We say that $D$ is a \textit{regular domain} if it is bounded and at least one assumption of the two holds:
\begin{itemize}
\item The domain $D$ is convex; or
\item Every point of $\partial D$ (topological boundary of $D$) has a neighbourhood $\mathcal{N}$ such that $\partial D \cap \mathcal{N}$ can be represented, up to a change of variables, as the graph of a $C^{1, 1}(\mathbb{R}^{d-1})$.
\end{itemize}

\end{definition}

We will always assume that a domain is regular, according to the above definition.
For a domain $D$, consider $\{ \phi_k \}_{k \in \mathbb{N}}$ be an orthonormal basis of $L^2(D)$ composed of eigenfunctions 
\begin{equation}\label{eq:defphik}
  \begin{cases}
    \Delta \phi_k = -\lambda_k \phi_k & \text{ in }D \\
	\phantom{\Delta} \phi_k = 0 & \text{ in }\partial D,
  \end{cases} 
\end{equation}
which are enumerated so that $\lambda_k$ is non-decreasing in $k$ .
Notice that by classical regularity theory, we have that $\phi_k$ are all
$C_c^\infty(D)$.
Remember that for all $k \ge 1$, $\lambda_k >0$. 

In the case of $D=\bb T^d$, besides requiring periodicity, we also require $\int_{D} \phi_k(x) dx =0$ rather than the boundary condition in \eqref{eq:defphik}. 
In this case, it will be more convenient to index the eigenfunctions and eigenvalues by $k \in \bb Z^d \setminus \{0\}$.
This is because the Fourier basis given by $\{\phi_k\}_{k \in \bb Z^d}$, where $\phi_k:= \exp (2\pi \iota k \cdot x)$ is also an orthonormal basis of eigenvectors of the Laplacian in $\bb T^d$ with periodic boundary conditions.
Again, because of the mean-zero condition, we exclude the $0$-th Fourier term and therefore $\lambda_k >0$. 

Due to this different indexing, in what follows, we will often write $\sum_{k \neq 0}(\dots)$ to denote either $\sum_{k = 1}^\infty(\dots)$ or $\sum_{k \in \bb Z^d \setminus \{0\}} (\dots)$ depending on whether we are working on a domain or in the torus.

We will denote by $\<f,g \>$  the $L^2(D)$ inner product if $f, g \in L^2(D)$
and (by abuse of notation) the pairing between a distribution $f$ and a smooth
function $g$. 

\begin{definition} \label{def:Hspaces}

For every dimension $d \ge 1$, and $\beta >0$, consider
\begin{itemize}
  \item For $D \subset \bb R^d$ a regular domain, we define the Hilbert space $H^\beta_0(D)$ as the closure of $C^\infty_c(D)$ functions with the norm
  \begin{equation}\label{eq:defNormH}
    ||f||^2_{H^\beta_0(D)}
      :=
      \sum_{k=1}^{\infty} |\< f,\phi_k \>|^2 \lambda_k^{2\beta}.
  \end{equation}
  \item  For $D =\bb T^d$, we define the Hilbert space $H^\beta_0(\bb T^d)$ as 
  \begin{equation}\label{eq:defHbeta}
    H^\beta_0(\bb T^d):=
    \left\{  f \in L^2(\bb T^d): \int_{\bb T^d} f(x) dx =0 \text{ and }||f||_{H^\beta_0(\bb T^d)}
    < \infty\right\}
  \end{equation}
  where 
  \begin{equation}\label{eq:defNormHtorus}
    ||f||^2_{H^\beta_0(\bb T^d)}
      :=
      \sum_{k \in \bb Z^d \setminus \{0\}} |\< f,\phi_k \>|^2 \lambda_k^{2\beta}.
  \end{equation}
\end{itemize}
\end{definition}

We can also easily characterize the dual space of $H^\beta_0(D)$.
For $\beta>0$, the space $H^{-\beta}_0(D)$ is given by the continuous linear functionals on $H^\beta_0(D)$.
The weak-$*$  topology in this space induces the norm given by \eqref{eq:defNormH}, but with $\beta$ substituted by $-\beta$.
To highlight whether we are talking about a space of functions or a space of distributions, unless stated otherwise, we take $\beta$ to be positive.

For strictly technical reasons, we will need to introduce another type of fractional Sobolev spaces denoted by $W^{\alpha,p}(D)$.
We will use these spaces in the proof of Lemma~\ref{lem:growthError} where we rely on the bounds provided by Theorem~\ref{thm:AKM} from \cite{Armstrong2019}.
We refer to Remark~\ref{rem:Sobolev} for a comment on the relation between the two types of Sobolev spaces.

For $m \in \bb N_0:=\bb N \cup \{0\}$ and $p \ge 1$, we will consider the Sobolev space $W^{m,p}(D)$ 
\begin{align*}\label{def:Wkp} 
	\|f\|_{W^{m, p}(D)}:=\sum_{0 \le |\alpha| \le  m}\|\partial^{\alpha} f\|_{L^{p}(D)},
\end{align*}
where for a given multi-index $\alpha$, we denote $\partial^\alpha f =
\partial^{\alpha_1}\dots \partial^{\alpha_d} f$ and $|\alpha|=\alpha_1+\cdots
+\alpha_d$. 
The general $m\in \mathbb{R}_+$ case will be defined as follows.
\begin{definition}
The \emph{fractional Sobolev space} $W^{\beta,p}(D)$ for any $\beta > 0$, $p\geq 1$ and $d\geq 2$ is defined by  
\begin{align*}
  W^{\beta, p}(D):=\{f \in W^{\lfloor \beta \rfloor, p}(D):\| f\|_{W^{\beta, p}(D)}<\infty\}
\end{align*}
where 
\begin{equation}\label{eq:defWalphap} 
  \|f\|_{W^{\beta, p}(D)}
  :=
  \|f\|_{W^{\lfloor \beta \rfloor, p}(D)}
  +
  \sum_{|\alpha|=\lfloor \beta \rfloor}[\partial^{\alpha} f]_{W^{\beta - \lfloor \beta \rfloor, p}(D)},
\end{equation}
$\lfloor \cdot \rfloor$ stands for the floor function and  
\begin{align*}
  [f]_{W^{\delta, p}(D)}^{p}:=(1-\delta) \int_{D} \int_{D} \frac{|f(x)-f(y)|^{p}}{|x-y|^{d+\delta p}} d x d y
\end{align*}
is called the Gagliardo semi-norm for $\delta \in (0,1)$.
\end{definition}

\begin{remark}[Comparison between $H^{\beta}_0(D)$ and $W^{\beta,p}(D)$ spaces]\label{rem:Sobolev}
For the purpose of this article, we are interested in the $W^{\beta,p}(D)$ topology in order to apply Theorem~\ref{thm:AKM}, for which the results are only valid for $p >2$.
Notice that in this case $W^{p,\beta}(D)$ is not a Hilbert space, therefore, we will not directly compare the topologies $W^{\beta,p}(D)$ and $H^{\beta}_0(D)$. 

However, the $W^{\beta,2}(D)$-norm of a mean zero function $f$ can be written in terms of the $L^2(D)$-norm of the singular integral definition of the fractional Laplacian operator restricted to $D$.
On the other hand, the $H^{\beta}_0(D)$-norm of $f$ is the $L^2(D)$-norm of the spectral definition of fractional Laplacian.
These spaces do not coincide in for general domains, unless $\beta \in \bb N$.
For more information on this matter, we suggest \cite{di2012hitchhikers}. 
\end{remark} 

Finally, for two non-negative functions $f,g$ we write 
\[
  f(x) \lesssim_{\gamma_1,\gamma_2, \dots ,\gamma_k} g(x)
\]
if there exists  a positive constant $C=C(\gamma_1, \gamma_2, \dots,\gamma_k)>0$ and we want to emphasize its dependence of the parameters
such that for all $x$
\[
  f(x) \leq C(\gamma_1,\gamma_2, \dots ,\gamma_k) g(x).
\]
If we do not want to emphasize the dependence we simply will write $f\lesssim g$.

\subsubsection{Discrete spaces}\label{sec:disspa}
Let $\bb T^d_N:= \frac{1}{N}
\bb Z^d_N$, where $\bb Z^d_N:= [-\frac{N}{2},\frac{N}{2})^d \cap \mathbb{Z}^d$ is
to be understood as a discretized torus, i.e. a graph with a periodic boundary. In this context,
we use $E^d_N$ to denote the edges of the graph $\bb Z^d_N$. 

For a function $f : \bb T^d_N \longrightarrow \bb C$, we will denote its 
restriction by $\bb T^d_N$ as $f^N = f|_{\bb T^d_N}$. 
Notice that 
$\{\phi^N_k\}_{k \in \bb Z^d_N}$ is an orthonormal basis 
of the space $\ell^2(\bb T^d_N):=\bb C^{\bb T^d_N}$ with respect to  the inner product 
\begin{align*}
  ( f,g ) 
  =
  ( f,g )_{\ell^2(\bb T^d_N)}
  :=
  \frac{1}{N^d} \sum_{\hat{x} \in \bb T^d_N} f(\hat{x})\overline{g(\hat{x})}. 
\end{align*}
When discussing scaling limits of  discrete models,
we will always use $x$, $y$ to denote points of $\bb Z^d$ or $\bb Z^d_N$, and use
$\hat{x}$, $\hat{y}$ to denote points in $\bb T^d_N$.

\subsection{Gaussian fields}\label{sec:hom-gauss}

In this section, we will introduce the Gaussian free field and bi-Laplacian field in continuous domains, $D\subset \mathbb{R}^d$.
First we will consider the homogeneous  and then  the non-homogeneous setting.
Throughout this article, we distinguish between two scenarios for $D$, in which we either take $D$ to be a regular domain (cf. Definition \ref{def:C11-domain}), or $D=\bb T^d$.

\subsubsection{Gaussian fields in the continuum}

\subsubsection*{Homogeneous fields}

\begin{definition}
The Gaussian free field $\Xi_D^g$  on $D$ is a random distribution
defined by 
\[
  \langle \Xi^g_D, f\rangle \sim N\left(0, \|f\|^2_{H^1_0(D)} \right)
\]
for $f\in C^{\infty}_0(\bb R^d)$ with covariance  given by
\[
  \mathbb{E} [ \<\Xi^g_D,f\>  \< \Xi^g_D,g\>]
  =
  \int_D \int_D f(x)g(x) G_D(x,y) dx dy
  \label{eq:defGFF}
\]
where $f,g \in C^{\infty}_0(\bb R^d)$. The Green's function $G_D(x,y)$ is given
by the distributional solution of 
\begin{equation} \label{eq:defGD}
  \begin{cases}
      -\left(\Delta G_D(x,\cdot) \right)(y) = \delta(x-y) & \text{ if } y \in D\\
      \phantom{\naeps}G_D(x,y) = 0 & \text{ if } y \in \partial D,
  \end{cases}
\end{equation}
where $x \in D$, $\Delta$ is the Laplacian operator and $\delta$ is the
Dirac delta function. For $D=\bb T^d$ we substitute the assumption on the
boundary condition in \eqref{eq:defGD} by $\int_{\bb T^d} G_{\bb T^d}(x,y)dy=0$. 
\end{definition}

A common representation of the Green's function in spectral terms is the following: 
\begin{equation} \label{eq:RepGD}
  G_D(x,y)
  :=
  \sum_{k = 1}^{\infty}
  \frac{\phi_k(x) \phi_k(y)}{\lambda_k}
\;
	\text{ and }
\;
  G_{\bb T^d}(x,y)
  :=
  \sum_{k \in \bb Z^d \setminus \{0\}}
  \frac{\phi_k(x) \phi_k(y)}{\lambda_k}.
\end{equation}
Let us remark again that, in the case $D=\bb T^d$, we can ignore the atom at $0$-th Fourier coeffient because we only consider mean-zero test functions.
We will use a similar representation for the non-homogeneous case, as it will allow to quantify the distance between the fields we are interested in.
Note that for any given test function $f$ we can write 
\[
\Xi^g_D(f)=\sum_{k \neq 0}\widehat{\Xi}^g_D(k)\< f,\phi_k\>, 
\]
where $\widehat{\Xi}^g_D(k):= \<\Xi^g_D,\phi_k\>$.
Therefore, we can reduce $\Xi^{g}_D$ to the infinite-dimensional Gaussian vector $\{\hat{\Xi}^g_D(k)\}_{k\geq 1}$ which has the covariance function
\begin{equation}
  \mathbb{ E} [\hat{\Xi}^g_D(k)  \hat{\Xi}^g_D(k^\prime)]
  =
  \frac{\delta_{k,k^\prime}}{\lambda_k},
  \label{eq:RepGFF}
\end{equation}
with $\delta_{k,k^\prime}$ representing the Kronecker delta function, and $k, k'\in \mathbb{N}$.
Let $\xi$ denote the (spatial) white noise in $L^2(\bb R^d)$: 
\begin{align} \label{eq:WhiteNoise}
  \nonumber
	\xi : L^2(\bb R^d)&  \longrightarrow L^2(\Omega) \\
	f & \longmapsto \<\xi,f\> 
	\sim N\left( 0,\|f\|^2_{L^2(\bb R^d)} \right).
\end{align}

In the following definition  we follow the approach in \cite{Lodhia2016}.

\begin{definition}
The bi-Laplacian field $\Xi^b_D$ in $D\subset\bb R^d$ is the random distribution, solution of
\begin{equation} \label{eq:defBiLap}
  \begin{cases}
	(-\Delta)\Xi^b_D =  \xi \text{ in } D \\
	\phantom{(-\Delta)}
	\Xi^b_D = 0 \text{ in } D^c,
  \end{cases} 
\end{equation}
where $\xi$ denotes the white noise defined in \eqref{eq:WhiteNoise}.
\end{definition}
The definition above can be related to the so-called \emph{eigenvalue bi-Laplacian field}, see \cite{Lodhia2016}.
In Section~\ref{sec:conclusion}, we discuss the another definition of bi-Laplacian field in a domain and how to obtain similar results.

Let us remark  that, for $\widehat{\xi}(k):= 
\<\xi,\phi_k\>$,  we have that 
\begin{equation} \label{eq:RepBiLap}
	\Xi^b_D :=G_D * \xi
	= \sum_{k \neq 0} \lambda^{-1}_k \widehat{\xi}(k) \phi_k
\end{equation}
\noindent
satisfies the distributional equation \eqref{eq:defBiLap}.
Moreover, if $D \subset \bb R^d$ is a domain due to Weyl's law  which states that $\lambda_k \asymp_{D,d} k^{2/d}$ (see \cite[Appendix B]{lax2002functional} for a proof), we have that $\Xi^{b}_D$ belongs to $H^{\beta}_0(D)$ for all $\beta < -\frac{d}{4}+1$.
Similarly, we can prove that $\Xi^{g}_D$ belongs to $H^\beta_0(D)$ for all $\beta < -\frac{d}{4}+\frac{1}{2}$.
Notice that for sufficiently small dimensions,  the fields can be interpreted as functions in positive Sobolev spaces.

\begin{remark} \label{rem:BoundaryCondGreen} 
Using the definition of the bi-Laplacian field above, one can show that its covariance function is given by
\begin{equation*}
\tilde{G}_D(x,y) = G_D*G_D(x,y) := \int_D G_D(z,x)G_D(z,y) dz
\end{equation*}
which, for each fixed $x \in D$,  solves the equation
\begin{equation*}
\begin{cases}
  \Delta^2 \tilde{G}_D (x,y)= \delta(x-y) & \text{ if  } y \in D,\\
  \phantom{\Delta^2}\tilde{G}_D(x, y)= 0 & \text{ if } y \in \partial D, \\
  \Delta \phantom{^2}  \tilde{G}_D(x, y)= 0, & \text{ if } y \in \partial D.
\end{cases} 
\end{equation*}
Notice that, as $\Delta^2$ is a differential operator of order $4$, we need to define additional boundary conditions when compared to partial differential equations driven by the Laplacian operator, in order to make $G_D$ well-defined.
This is the so-called Green's function of the bi-Laplacian operator with Navier boundary conditions.
There are other sensible choices for the boundary conditions for the bi-Laplacian operator on a regular domain $D$.

One can then define the bi-Laplacian field by setting its covariance to be the Green's function of a bi-Laplacian operator with a given boundary condition (similar to Definition~\ref{def:homGFF}).
However, one needs to consider that this choice leads to different fields.

In contrast, on the torus, periodic boundary conditions are the only natural choice for the Green's function of the bi-Laplacian field (and therefore for associated field). 
\end{remark}

\subsubsection*{Gaussian fields in random media} 

\noindent
Let us introduce the random environment {\bf a} and present some assumptions.
We will follow the approach in \cite{Armstrong2019}.
The probability space of the environment will be denoted by $(\Omega(\Lambda), \tilde{\mc F}, {\bf P})$, which we will define in the sequel.

Call $\bb R^{d\times d}_{\sym}$ the set of real valued symmetric matrices.
We will be interested in maps ${\bf a}: \bb R^d \to \bb R^{d \times d}_{\sym}$ for which there is a constant $\Lambda>1$ such that
\begin{equation} \label{eq:UniElip}
	\|v\|^2 \le v \cdot \mathbf{a}(x) v \le \Lambda \|v\|^2,
	\forall v,x \in \bb R^d.
\end{equation}
We then define the space 
\begin{equation}
	\Omega(\Lambda):= \{ {\bf a} \in \bb R^{d\times d}_{\sym}: 
	{\bf a} \text{ Lebesgue measurable  and satisfying } \eqref{eq:UniElip}  \}.
	\label{eq:OmLamb}
\end{equation}

Fix $U \subset \bb R^d$, we will consider  $\mc F_U$, the $\sigma$-algebra generated by the maps 
\begin{equation*}
	{\bf a}\longmapsto \int_U {\bf a}_{ij}(x) f(x) dx
\end{equation*}
for $i,j \in\{1,\dots,d\}$ and $f \in C^\infty_c(U)$ a test function.
Moreover, we define $\tilde{\mathcal{F}} =\mc F_{\bb R^d} := \sigma \left(\bigcup_{U \subset \bb R^d} \mc F_U \right)$.
Let ${\bf P}$ be a probability distribution on $(\Omega(\Lambda),\tilde{\mc F})$ with the following properties:
\begin{enumerate}
	\item[(A1)] \textbf{Translation invariance}: For each $z \in \bb R^d$, we have ${\bf P} \circ (\tau_z)^{-1}= {\bf P}$ where $\tau_z {\bf a}(x):={\bf a}(x+z)$ is a translation.
	\item[(A2)] \textbf{Unit range of dependence}: $\mc F_U$ and $\mc F_V$ are ${ \bf P}$-independent for each pair $U,V \subset \bb R^d$ such that $d(U,V)\ge1$, $d(\cdot,\cdot)$ is the Hausdorff distance.
	\item[(A3)] \textbf{Isotropic in law}: For each $I: \bb R^d \longrightarrow \bb R^d$ isometry such that $I$ maps each of the coordinate axes to another, we have that ${\bf P} \circ I^{-1} = {\bf P}$.
\end{enumerate}
Let us remark that Assumptions $(A1)$ and $(A2)$ are standard in \cite{Armstrong2019}, but $(A2)$ can be relaxed to sufficiently fast decaying covariances.
Assumption $(A3)$ is a matter of convenience. In this case, the homogenized operator is given by a multiple of the Laplacian (instead of $\nah$ where $\ahom$ is a matrix).
In this case, the asymptotic behaviour of the $L^p(\mathbb{R}^d)$-norms of the eigenvectors of the limiting operator, $\phi_k$'s , are well understood.

\begin{definition} \label{def:GFFNHom}
  Let $(\Omega(\Lambda) \otimes \Omega, \tilde{\mathcal{F}}\otimes \mc F, {\bf
  P} \otimes {\bb P})$,  $D \subset \bb R^d$ a regular domain, and  ${\bf
  a} \in \Omega(\Lambda)$ fixed. Furthermore, let $G_D^{{\bf a}}(\cdot,\cdot)$ be the
  solution of
  \[
  \begin{cases}
    (-\nan G^{{\bf a}}_D(x, \cdot))(y)  = \delta(x-y)& \text{ if } y \in D\\
    \phantom{(-\nan} G^{{\bf a}}_D(x,y)\phantom{\cdot(y)}   = 0& \text{ if } y\in \partial D,
  \end{cases}
  \]
  for each $x \in D$. 
  We define 
  $\Xi^{g,{\bf a}}_{D}$ a \emph{non-homogeneous Gaussian free field} 
  as the Gaussian random field with mean 0 and covariance given by 
  \begin{equation} \label{eq:GFFNHom}
	  {\bb E}[ \<\Xi^{g,{\bf a}}_{D},f\>
	  \<\Xi^{g,{\bf a}}_{D},g\>  ]
	  = \int_D \int_D G_D^{\bf a}(x,y) f(x)g(y)dx dy,
  \end{equation}
  where $f,g\in C^{\infty}_c(\bb R^d)$. 
\end{definition}

\begin{definition} \label{def:BiLapNHom}
Let $(\Omega(\Lambda) \otimes \Omega, \tilde{\mathcal{F}}\otimes \mc F, {\bf P} \otimes {\bb P})$, $D \subset \bb R^d$ a regular domain, and ${\bf a} \in \Omega(\Lambda)$ fixed.
We define $\Xi^{b, \bf a}_D$ a \emph{non-homogeneous bi-Laplacian} as the distributional solution to
  \begin{equation} \label{eq:BiLapNHom}
    \begin{cases}
      -	\nan \Xi^{b, {\bf a}}_D
	  =
	  \xi \text{ in } D. \\
\phantom{\;\;\;\;\nan} 
	 \Xi^{b, \bf a}_D = 0 \text{ in } D^c.
    \end{cases} 
  \end{equation}
\end{definition}

For a given random environment ${\bf a} \in \Omega(\Lambda)$, stochastic homogenization techniques will allow determining  limiting fields of $\Xi^{g, {\varepsilon}}_D:=\Xi^{g, {\bf a}_{\varepsilon}}_D$ resp.
$\Xi^{b, {\varepsilon}}_D:=\Xi^{b, {\bf a}_{\varepsilon}}_D$ where ${\bf a}_{\varepsilon}(x) :={\bf a}(x/\varepsilon)$ and $x \in D$.

The main idea of this article is to obtain a good representation for the non-homogeneous fields not in terms of the eigenfunctions of $\naeps$ but instead, in terms of the eigenfunctions of the liming operator.
 For this, we will need to introduce a sequence of functions that we call the \emph{pseudo-eigenfunctions} $\{\varphi_k^{\varepsilon}\}_{k\in \mathbb{N}}$,  solutions of
\begin{equation}
	\label{eq:NHPDEBasis}
	\begin{cases}
		-\naeps \varphi^{\varepsilon}_k = \lambda_k \ahom \phi_k & \text{ in } D\\
	\phantom{\,\,\;\;\;\;\nan} 
		\varphi^{\varepsilon}_k= 0 & \text{ in } \partial D.
	\end{cases}
\end{equation}
The deterministic constant $\ahom$ is positive and has a special meaning in stochastic homogenization where it is referred to as \emph{effective/homogenized coefficient}, see Remark~\ref{rem:effective} for more details.

In particular, we will show that the non-homogeneous GFF $\Xi_D^{g, \varepsilon}$ can be written as a mean zero infinite Gaussian vector
\begin{equation}\label{eq:repGFFeps}
  \Xi_D^{g,{\varepsilon}}= \sum_{k \neq 0} \hat{\Xi}_D^{g,{\varepsilon}}(k) \phi_k
\end{equation}
where $\hat{\Xi}_D^{g,{\varepsilon}}(k):=\< \Xi_D^{g},\varphi^{\varepsilon}_k\>$, and has covariance structure
\begin{equation} \label{eq:RepCovGffEps}
  {\bb E} [ \hat{\Xi}_D^{g,{\varepsilon}}(k)\hat{\Xi}_D^{g,{\varepsilon}}(k')]
  =
  \frac{\<\phi_k,\varphi^\varepsilon_{k^\prime}\>}{\ahom \lambda_{k^\prime}},
\end{equation}
for $k,k^\prime \neq 0$. 
Notice that as $\nabla \cdot {\bf a}_{\varepsilon} \nabla$ is self-adjoint, we have for all $k,k' \neq 0$
\begin{equation}\label{eq:pseudo}
  \langle \phi_k, \varphi^{\varepsilon}_{k'} \rangle =
  \frac{\lambda_{k'}}{\lambda_k} 
  \langle \varphi^{\varepsilon}_{k},\phi_{k'} \rangle
  \, \, \text{ and } 
  \langle \naeps
  \phi_k, \varphi^{\varepsilon}_{k'} \rangle =\delta_{k,k'}.
\end{equation}

Similarly, for the non-homogeneous bi-Laplacian field
\begin{equation}\label{eq:RepXieps}
    \Xi^{b, \varepsilon}_D 
    =
    \sum_{k \neq 0} 
    (\bar{{\bf a}}\lambda_k)^{-1} \hat{\xi}^{\varepsilon}(k) \phi_k
\end{equation}
where $\hat{\xi}^{\varepsilon}(k)=\langle \xi, \varphi^{\varepsilon}_k\rangle$ and $\varphi_k^{\varepsilon}$ is defined in \eqref{eq:NHPDEBasis}.

\begin{remark}[The effective coefficient $\ahom$ ]\label{rem:effective} 
The effective coefficient $\ahom$ only depends on the law of the environment ${\bf P}$.
In particular in \cite[Equation 3.92]{Armstrong2019}, the authors provide an explicit formula for $\ahom$ in terms of the expected energy and flux of the first-order correctors $\Phi$.
That is, for any unit vector $e\in \mathbb{R}^d$:
\begin{equation}\label{def:a}
  \ahom 
  =
  {\bf E} \left[ \int_{[0,1]^d}
  \frac{1}{2}(e+\nabla \Phi_e(x))\cdot {\bf a}(e+\nabla \Phi_e(x)) dx \right]
\end{equation}
and the first-order corrector $\Phi_e$ is the difference between the solution of the Dirichlet problem and an affine function defined in \cite[Section 3.4]{Armstrong2019}.
A similar characterisation holds in the discrete setting (which we introduce in the next section).
For more information see for instance \cite{gloria2011optimal}.
\end{remark}
\subsubsection{Gaussian fields in discrete spaces}
\label{sec:nonh-gauss-discrete}

\subsubsection*{Homogeneous fields}
\noindent
Let $\Delta_N$ be the normalized graph 
Laplacian of $\bb T^d_N$ (the discrete torus defined in Section~\ref{sec:disspa}), that is
\begin{equation}\label{eq:discLap}
  \Delta_N f (\hat{x})
  =
  N^2
  \sum_{ \hat{y}; \hat{y}  \sim\hat{x}}
  (f(\hat{y})-f(\hat{x})),
\end{equation}
where $\hat{y}\sim \hat{x}$ denotes that $\hat{x}$ and $\hat{y}$ 
are nearest neighbours in $\bb T^d_N$. Note that we will use the same definition of the normalized discrete Laplacian as \cite{gloria2014optimal}.
In this case, we take the Green's function as the solution of 
\begin{align*}
  \begin{cases}
    (-\Delta_N) G^N(\hat{x},\cdot)(\hat{y}) 
    =
    \delta_{\hat{x},\hat{y}}- \frac{1}{N^d} & \text{ in }  \bb T^d_N \\
    \sum_{\hat{y} \in \bb T^d_N} G^N(\hat{x},\hat{y}) \phantom{di} = 0.
  \end{cases} 
\end{align*}
It will be particularly useful for us 
that $\{\phi_k^N\}_{k \in \bb Z^d_N}$  is also a basis
of eigenfunctions of $\Delta_N$, remember that $f^N:=f|_{\bb T^d_N}$.
Moreover, we have that 
the Green's function of the discrete torus $\bb T^d_N$ can 
be written as 
\begin{equation} \label{eq:defGdTorus} 
  G^N( \hat{x}, \hat{y})
  :=
\frac{1}{2d}  \frac{1}{N^{d}}
  \sum_{k \in \bb Z^d_N \setminus\{0\}}
 \frac{\phi_k(\hat{x}-\hat{y})}{\lambda^{(N)}_k},
\end{equation}
where $\lambda^{(N)}_k$ is the eigenvalue of $\Delta_N$ associated with
$\phi^N_k$.

Before defining the discrete versions of Gaussian fields in random environments, we remind the reader of the definitions of the discrete GFF and the discrete bi-Laplacian field in a discrete torus.

\begin{definition}\label{def:homGFF}
The discrete Gaussian free field on the torus is a multivariate Gaussian vector $(\Xi^g_N(\hat{x}))_{\hat{x} \in \bb T^d_N}$ with mean zero and covariance ${\bb E}[\Xi^g_N(\hat{x})\Xi^g_N(\hat{y})]=G^N(\hat{x},\hat{y})$, where $G^N$ is defined in \eqref{eq:defGdTorus}.
\end{definition}

\begin{definition}
Let $(\xi_N(\hat{x}))_{\hat{x} \in \bb T^d_N}$ be a collection of i.i.d random variables with distribution $N(0,1)$.
The discrete bi-Laplacian field $(\Xi^b_N(\hat{x}))_{\hat{x} \in \bb T^d_N}$ is the solution of the finite difference equation
\begin{equation}\label{eq:defEtaN}
 \begin{cases}
   (- \Delta_N) \Xi^b_N(\hat{x}) \, 
   =
   \xi_N(\hat{x}) - (\xi)_N& \text{ in }  \bb T^d_N \\
   \sum_{\hat{x} \in \bb T^d_N} \Xi^b_N(\hat{x}) =0.
 \end{cases}  
\end{equation}
where $(\xi)_N:=\frac{1}{N^d}\sum_{\hat{x}\in \bb T^d_N} \xi_N(\hat{x})$ is the spatial average of $\xi_N$ on $\bb T^d_N$. 
\end{definition}

\subsubsection*{Gaussian fields with random conductances}

\noindent
In this context, we will use the stochastic homogenization bounds from \cite{gloria2014optimal} instead of \cite{Armstrong2019}.
We will keep some of the same notation used in the continuous case in order to make a clear analogy between the two.
The underlying probability space will be denoted by $(\Omega^{\dis}(\Lambda), \tilde{\mc F}, {\bf P})$.

Consider the Euclidean lattice graph given by $(\bb Z^d, E^d)$, where $\{x,y\}
\in E^d$ if $\|x-y\|_1=1$. Denote the measurable space given by
$(\Omega^{\dis}(\Lambda), \tilde{\mc F})$, where
\[
  \Omega^{\dis}(\Lambda)
  :=
  \{{\bf a}=(\mathbf{a}_e)_{e \in E^d}:1 < \mathbf{a}_e\le \Lambda \}
\]
for some arbitrary constant $\Lambda >1$, and $\tilde{\mc F}$ is the product $\sigma$-field indexed by the edges.  

\begin{remark}
Notice that in this discrete context, we are associating the random environment ${\bf a}$ to each individual edge $e \in E^d$, rather than defining a matrix-valued function (as in the continuous case).

Given an element ${\bf a} \in \Omega^{\dis}(\Lambda)$, we can recover a $2d \times 2d$ matrix-valued function that gives weights to each of the $2d$ discrete differences of a function $f$.
However, it is more convenient to simply index the environment ${\bf a}$ via the edges, particularly when we want to define the projections $\Pi_{N}$ below.

Notice that because edges $e \in E^d$ are always parallel to the axes, we never see mixed derivatives - which are allowed in continuous stochastic homogenization.
The hypothesis ${\bf a}_e \in (1,\Lambda]$ is a discrete version of uniform ellipticity and continuity  assumption, cf. \eqref{eq:UniElip}.

Product measures on $\Omega^{\dis}(\Lambda)$ should be seen as natural examples of measures satisfying assumptions (A1)-(A3) from the continuous case.
In fact, we believe our results should also hold when substituting the product measures with measures which are invariant by translations and rotations, provided that their covariances decay fast enough.
\end{remark}

Similarly to the infinite case, we define $(\Omega^{\dis}_N(\Lambda),\tilde{\mc F}_N)$ by substituting $\bb Z^d$ with $\bb Z^d_N$ and $E^d$ with $E_N$.
Let $\bf P$ be any product probability measure on $\Omega^{\dis}(\Lambda)$.

We define a projection operator  $\Pi_N: \Omega^{\dis}(\Lambda) \longrightarrow
\Omega^{\dis}_N(\Lambda)$ 
by
\[
  (\Pi_N \mathbf{a})_{e}
  :=
  \begin{cases}
    \mathbf{a}_{\{x,x+e_i\}}
    & \text{ if } e = \{x,x+e_i\} \in E^d_N, i \in \{1,\dots, d\} \\
    \mathbf{a}_{\{x,x+e_i\}}
    & \text{ if } e = \{x,x-N e_i\}\in E^d_N, i \in \{1,\dots, d\},
  \end{cases}
\]
where $(e_i)_{i=1}^d$ is the canonical basis of $\bb R^d$, 
$\mathbf{a} \in \Omega^{\dis}(\Lambda)$ and $e \in E^d_N$.

On the other hand, we define an extension operator from
$\Omega^{\dis}_N(\Lambda)$ to $\Omega^{\dis}(\Lambda)$ for $\mathbf{a}\in
\Omega_N^{\dis}(\Lambda)$ and $e = \{x,y\} \in E^d$ as
\[
  (\ext_N {\mathbf{a}})_e
  := {\mathbf{a}}_{e^*},
\]
where $e^*$ is the only edge in $E^d_N$ such that $e^* = \{x^*,y^*\}$ with $x^* \equiv x \mod \bb Z^d_N$ and $y^* \equiv y \mod \bb Z^d_N$.
For $\mathbf{a} \in \Omega^{\dis}(\Lambda)$, we write $\mathbf{a}_N := \ext_N \circ \Pi_N(\mathbf{a})$.
Notice that,  $\Pi_N \circ \ext_N$ is the identity in $\Omega^{\dis}_N(\Lambda)$ and $\ext_N \circ \Pi_N(\mathbf{a})_e = \mathbf{a}_e$ for all $\mathbf{a} \in \Omega^{\dis}(\Lambda)$ whenever $e = \{x,y\}$ with $x,y \in \bb Z^d_N$.
Let ${\bf P}_N$ be an abbreviation for ${\bf P}(\Pi_N^{-1} \cdot)$, which is a product measure in $\bb Z^d_N$.

Finally, for any $f: \bb T^d_N \longrightarrow \bb R$ and
any $\mathbf{a} \in \Omega^{\dis}(\Lambda)$, we can write

\begin{equation}\label{eq:def-random-operator-discrete}
	-\naN  f (\hat{x})
	:= N^2 \sum_{\hat{y}:\hat{y} \sim \hat{x} } \mathbf{a}_{N, \{\hat{x},\hat{y}\}} (f(\hat{x})-f(\hat{y})).
\end{equation}

\begin{definition}\label{def:gn}
Let $(\Omega^{\dis}(\Lambda) \otimes \Omega, \tilde{\mathcal{F}}\otimes \mc F,
\bf P \otimes {\bb P})$.
We define the discrete non-homogeneous Gaussian free field to be the Gaussian
vector $(\Xi_N^{g,{\bf a}}(\hat{x}))_{\hat{x} \in \bb T^d_N}$  with mean $0$ and
covariance
\[
  {\bb E}[\Xi_N^{g,{\bf a}}(\hat{x})\Xi_N^{g,{\bf a}}(\hat{y})]
  :=
  G^{N,{\bf a}}(\hat{x},\hat{y}),
\]
where  $G^{N,{\bf a}}: \bb T^d_N \times \bb T^d_N \longrightarrow \bb R$ is the
unique solution of
\begin{align} \label{def-Green-function}
  \begin{cases}
    (-\nabla \mathbf{a}_N \nabla G^{N,{\bf a}} (\cdot,\hat{y}))(\hat{x})
    =
    \delta^N_{\hat{x},\hat{y}}- \frac{1}{N^d}, & \hat{x} \in \bb T^d_N  \\
	\phantom{\;\;\,} 
    \sum_{\hat{x} \in \bb T^d_N}G^{N,{\bf a}}(\hat{x},\hat{y}) \;\;\;\;\;= 0,
  \end{cases} 
\end{align}
and $\delta^N_{\hat{x},\hat{y}}= \sum_{{z} \in \bb Z^d}
\delta_{N\hat{x},N\hat{y}+z}$, with $\delta_{\hat{x},\hat{y}}$ being the
standard Kronecker delta function.

\end{definition}
Call $(\xi_N(\hat{x}))_{\hat{x}\in  \bb T^d_N}$  the i.i.d collection of
standard normal random variables. 

\begin{definition}\label{def:bn}
Let $(\Omega^{\dis}(\Lambda) \otimes \Omega, \tilde{\mathcal{F}}\otimes \mc F,
\bf P \otimes {\bb P})$. The collection of Gaussian random variables 
$(\Xi_N^{b,{\bf a}}(\hat{x}))_{\hat{x} \in \bb T^d_N}$  
satisfying
\begin{equation}\label{eq:discBiLap}
  \begin{cases}
    -\naN \Xi_N^{b,{\bf a}}(\hat{x})
    =
    \xi_N(\hat{x}) - (\xi)_N, & x \in  \bb T^d_N \\
	\phantom{\;\;\;\;\;\;\;\;\,} 
    \sum_{\hat{x}\in \bb T^d_N} \Xi_N^{b,{\bf a}}(\hat{x})=0
  \end{cases} 
\end{equation}
is called \emph{discrete non-homogeneous bi-Laplacian field}. Again, we are using
the notation $(\xi)_N:=\frac{1}{N^d}\sum_{\hat{x}\in \bb T^d_N} \xi_N(\hat{x})$.

\end{definition}

The formal field on $D=\bb T^d$ as for $i\in \{b,g\}$ is defined by
\begin{equation}\label{def:formalField}
  \Xi^{i, {\bf a}}_{D,N} 
  :=
 \frac{c_i}{N^{d/2}} \sum_{\hat{z}\in \bb T_N^d} \Xi^{i, {\bf a}}_N(\hat{z}) \delta_{\hat{z}}
\end{equation}
where $\Xi^{g, {\bf a}}_N$ is defined in Definition \ref{def:gn} and $\Xi^{b,
{\bf a}}_N$ in Definition \ref{def:bn} and $c_g=(2d)^{-1/2}$ resp. $c_b=(2d)^{-1}$.

\begin{remark} \label{rem:abuse} 
	In \eqref{def:formalField}, there is a clear abuse of notation as we are using $\Xi^{i,{\bf a}}_N(\cdot)$ to denote both a random vector (which has well-defined values for every choice of $\hat{z}$) and a distribution given by the sum of delta functions (and therefore has no well-defined notion of value at a given point).
	We chose to do so to simplify the notation.
\end{remark} 

\begin{figure}[!htb]
  \begin{subfigure}[b]{0.5\linewidth}
    \centering
    \includegraphics[width=0.7\linewidth]{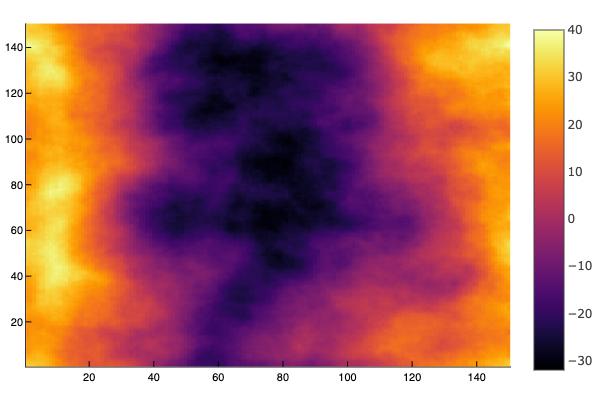}
    \caption{${\bf a}_N(e)\equiv 3/2$, constant environment }
    \vspace{4ex}
  \end{subfigure}
  \begin{subfigure}[b]{0.5\linewidth}
    \centering
    \includegraphics[width=0.7\linewidth]{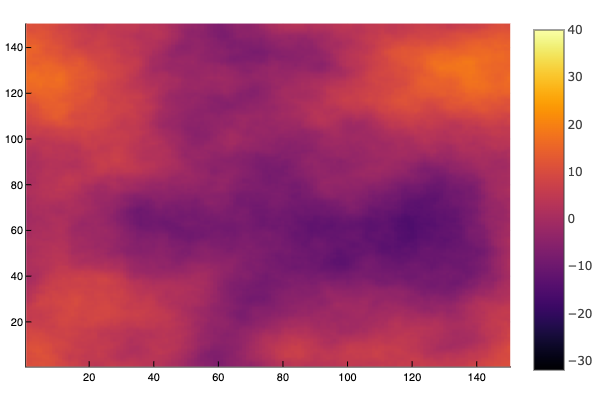}
     \caption{${\bf a}_N(e) \sim Unif(1,2)$}
    \vspace{4ex}
  \end{subfigure}
  \begin{subfigure}[b]{0.5\linewidth}
    \centering
    \includegraphics[width=0.7\linewidth]{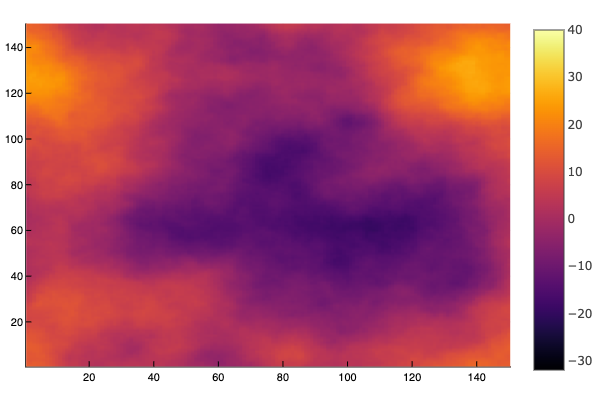}
    \caption{${\bf a}_N(e) \sim 1+Ber(0.5)$}
  \end{subfigure}
  \begin{subfigure}[b]{0.5\linewidth}
    \centering
    \includegraphics[width=0.7\linewidth]{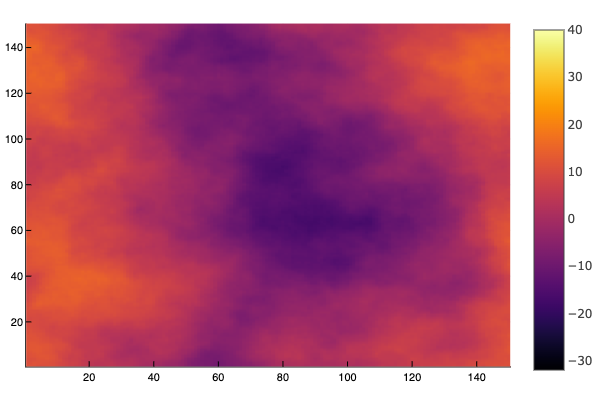}
    \caption{${\bf a}_N(e) \sim 1+Ber(0.5)$}
  \end{subfigure}
  \caption{Heat map simulations of the discrete bi-Laplacian fields  for $N=150$
    samples with  the same noise but different environments ${\bf a}_N$. 
    There is a significant change in local features across the
    different simulations, even between (c) and (d), which are sampled according
    to the same law. However, the location of the maxima and minima
    is consistent through out all samples.
  }
  \label{Fig1}
\end{figure}

\section{Main results}\label{sec:results}

\begin{theorem}\label{thm:ConvCont}
    Let $D \subset \bb R^d$ be a regular domain,
    $(\Omega(\Lambda) \otimes \Omega, \tilde{\mathcal{F}}\otimes \mc F, \bf
    P \otimes {\bb P})$ the underlying probability space, where ${\bf P}$ satisfies Assumptions (A1)-(A3). Then
    there exists a positive deterministic constant $\ahom$ such that the
    following statements hold.
    \begin{enumerate}
    \item  The non-homogeneous GFF defined in \eqref{eq:repGFFeps} converges
    \[
        \Xi^{g,\varepsilon}_D \overset{d}\longrightarrow \bar{{\bf a}}^{-1/2} \Xi^g_D,
    \]
      as $\varepsilon \to 0$    in $H^{-\beta}_0(D)$, with $\beta>\frac{d}{4}$.
    \item The non-homogeneous bi-Laplacian field defined in \eqref{eq:RepXieps}
      converges
      \[
        \Xi^{b,\varepsilon}_D \overset{{\bf P}\otimes {\bb P}}\longrightarrow \bar{{\bf a}}^{-1} \Xi^b_D,
      \]
      as $\varepsilon \to 0$    in $H^{-\beta}_0(D)$, with
      $\beta>\frac{d}{4}-\frac{1}{2}$.
    \end{enumerate}
\end{theorem}

\begin{theorem}\label{thm:ConvDisc}
   Let $(\Omega(\Lambda) \otimes \Omega, \tilde{\mathcal{F}}\otimes \mc F, \bf P
   \otimes {\bb P})$ where ${\bf P}$ is a product measure. Then there exists a
   positive deterministic constant $\ahom$ such that (1) and (2) hold.
   \begin{enumerate}
    \item The discrete non-homogeneous GFF defined in \eqref{def:formalField}
      satisfies
    \[
        \Xi^{g, {\bf a}}_{D,N} \overset{d}\longrightarrow 
	\bar{{\bf a}}^{-1/2} \Xi^g_{\bb T^d}
      \]
    as $N \to \infty$.
	The convergence holds in $H^{-\beta}_0(D)$, $\beta >\frac{d}{4}$.
\item There exists a coupling between $\xi$ and the sequence $\xi_N$, such that the discrete non-homogeneous bi-Laplacian field defined in \eqref{def:formalField} satisfies
      \[
        \Xi^{b, {\bf a}}_{D,N} \overset{\bf P \otimes {\bb P}}\longrightarrow 
	\bar{{\bf a}}^{-1} \Xi^b_{\bb T^d},
      \]
      as $N \to \infty$. The convergence holds in $H^{-\beta}_0(D)$, $\beta
      > \frac{d}{4}-\frac{1}{2}$.
      \end{enumerate}
\end{theorem}

For more details on $\ahom$ in both contexts, see Remark~\ref{rem:effective}.
Recall that, by abuse of notation, we use $\ahom$ simultaneously for $\ahom = c \Id$, where $\Id$ is the $d \times d$ identity matrix and $c=\ahom \in \bb R_{+}$ the deterministic constant that only depends on the law of the environment. 

Note that in Theorem \ref{thm:ConvDisc} we only take the limit in $N\rightarrow \infty$ to obtain the homogenized limiting field.
This is due to the translation invariant nature of the environment ${\bf a}$.

Finally, let us mention that one could easily prove that non-homogeneous bi-Laplacian fields converge almost surely (in more irregular Sobolev spaces than the ones stated above) to their homogeneous counterpart by a simple Borel-Cantelli argument.

\section{Proofs}
\label{sec:proofs}

\subsection{Proof of Theorem \ref{thm:ConvCont}}
\label{subsec-proof-bilap}
We start this section by stating some necessary results regarding stochastic homogenization.
Then we proceed to prove some useful lemmata necessary to recover our representation of the Green's functions before finally proving the main theorems.

Under the Assumptions (A1)-(A3), there are good bounds for stochastic homogenization of the solutions of  elliptic partial differential equations.
Indeed, one can prove that there is a positive constant $\ahom$ (see Remark~\ref{rem:effective}),   such that for $u$, any sufficiently regular function, we can estimate its $L^2(D)$ distance to the function $u^\varepsilon$ given by 
\begin{equation}
	\begin{cases}
		-\naeps u^\varepsilon = \ahom \Delta u & \text{ in } D,\\
		\;\;\;\;
		\phantom{\naeps}u^\varepsilon = 0 & \text{ in } \partial D.
	\end{cases}
	\label{eq:NHPDE}
\end{equation}
We will state a simplified version of a result from \cite{Armstrong2019} which will be one of the main ingredients for our proofs.

For this, we use the stochastic integrability notation $\mc O_s(\cdot)$ according to the law ${\bf P}$.
Given a random variable $X$ and parameters $s,\theta \in (0,\infty)$, we say that $X=\mc
O_s(\theta)$ if
\begin{equation}\label{eq:NotO}
  {\bf E}[\exp ((\theta^{-1} X_{+})^{s})] \le 2,
\end{equation}
where $a_+:=\max\{0,a\}$. 

\begin{theorem}[Theorem 6.17 in \cite{Armstrong2019}] \label{thm:AKM}
Under Assumptions $(A1)-(A3)$, fix  $s \in (0,2)$, $\alpha \in(0,1], p \in(2, \infty],$ a regular domain $D \subseteq \mathbb{R}^{d}$ and $\varepsilon \in(0, \frac{1}{2}]$.
There exists $\delta=\delta(d,\Lambda)\in (0,\infty)$, $C=C(s,\alpha,p,D,d,\Lambda) \in (0,\infty)$ and a  non-negative random variable $\mc X_{\varepsilon}$ satisfying the following estimate
  \begin{equation}
	  \mc X_\varepsilon
	  \le
	  \begin{cases}
		  \mc O_{s/2}(C\varepsilon^{\alpha}( \log (1/\varepsilon) ) ) 
		  & \text{ for }d=2\\
		  \mc O_{1+\delta}(C\varepsilon^{\alpha} ) 
		  & \text{ for }d\ge 3\\
	  \end{cases}
	  \label{eq:StochOX}
  \end{equation}
 such that the following holds: for every $u \in W^{1+\alpha, p}(D),$ and $u^{\varepsilon} \in H^{1}_0(D)$ be the solution of \eqref{eq:NHPDE} then we have that %
  \begin{equation}
	  \label{eq:thmAKM}
	  \|u^{\varepsilon}-u\|_{L^2(D)}
	  \le \mc X_{\varepsilon}\|u\|_{W^{1+\alpha, p}(D)} .
  \end{equation}
\end{theorem}

The theorem above implies that  for each fixed $k$, as $\varepsilon \to 0$,  the function $\varphi^\varepsilon_k$ defined in \eqref{eq:NHPDEBasis} converges to $\phi_k$ in $L^2(D)$ in probability.
In the next lemma, we will quantify such convergence in terms of $k$ and $\varepsilon$.

\begin{lemma}
  \label{lem:growthError}
  On $(\Omega, \mathcal{F}, \mathbb{P})$ we have that for all $\alpha \in (0,1]$, 
  and $\kappa>0$
  \begin{equation}
    \|\varphi_k^\varepsilon-\phi_k\|_{L^2(D)}\lesssim_{D,\alpha,\kappa}
    {\mc X}_\varepsilon \lambda_k^{\frac{1+\alpha}{2}+\kappa}
    \label{eq:growthError}
  \end{equation}
  where $\mc X_\varepsilon$ satisfies the bound \eqref{eq:StochOX} for a suitable choice
  of parameters.
\end{lemma}

\begin{proof}
Fix $\alpha \in (0,1], \kappa>0$ and let $p = 2 +\delta$ with $\delta>0$ to be defined later to be sufficiently small.
By applying Theorem~\ref{thm:AKM} to $u=\phi_k$, where $\phi_k$ is the eigenfunction associated to $\lambda_k$ the $k$-th eigenvalue of $(-\Delta)$ we have that for $p \in (2,\infty)$,
\begin{equation}
  \|\varphi^\varepsilon_k-\ahom\phi_k\|_{L^2(D)}
  \lesssim  \mc X_{\varepsilon}\|\phi_k\|_{W^{1+\alpha, p}(D)}.
  \label{eq:AKMphik}
\end{equation}
By \cite[Theorem 6.4.5]{bergh2012interpolation}, for any given a function $f \in C^\infty_c(\bb R^d)$, $s_1,s_2 \in \bb R$, $p_1,p_2 > 1$ and $\theta \in (0,1)$, we have 
\begin{equation}
  \|f\|_{W^{s_\theta,p_\theta}(\bb R^d)}
  \lesssim_{p_1,p_2,s_1,s_2,\theta,d}
  \|f\|^{1-\theta}_{W^{s_1,p_1}(\bb R^d)}
  \|f\|^\theta_{W^{s_2,p_2}(\bb R^d)}
  \label{eq:InterSob}
\end{equation}
as long as 
\begin{equation}
  \frac{1}{p_\theta}
  =
  \frac{1-\theta}{p_1}
  +
  \frac{\theta}{p_2}
  \text{ and }
  s_\theta
  =
  (1-\theta)s_1
  +
  \theta s_2.
  \label{eq:CondInter}
\end{equation}
This can also be extended to any such function $f \in W^{s_\theta,p_\theta}(D)$.
Now, if we choose $s_\theta = 1+\alpha$, $p_\theta = 2+\delta$, $s_1=p_1=2$ and $s_2=0$, we have that  $\theta = (1-\alpha)/2$ and 
\[
  p_2=p_2(\delta)
  :=\frac{2(2+\delta) (1-\alpha)}{4-(2+\delta)(1+\alpha)}.
\]
Notice that as $\delta \longrightarrow 0^+$, we have that $p_2 \longrightarrow 2^+$ for any $\alpha \in (0,1]$.
Furthermore, by classical results any $f \in C_c^\infty(\bar{D})$,  such that $f\mid_{\partial D}\equiv 0$ can be naturally extended to
\begin{equation*}
  \ext f(x) :=
  \begin{cases}
  	f(x), & x \in D \\  
	0, & x \not \in D,
  \end{cases}
\end{equation*}
so that $\|\ext f\|_{L^{p_2}(\mathbb{R}^d)}=\| f\|_{L^{p_2}(D)}$ and $\|\ext f\|_{W^{2,2}(\mathbb{R}^d)}=\| f\|_{W^{2,2}(D)}$.

As each of the $\phi_k$ are in $C_c^\infty(D)\cap H^{2,2}_0(D)$, we can use \eqref{eq:InterSob}, and the isometry properties above to get
\begin{align*}
  \|\phi_k\|_{W^{1+\alpha,p}(D)}
  &\le
  \|\ext \phi_k\|_{W^{1+\alpha,p}(\bb R^d)}\\
  &\lesssim_{\alpha,p_2,d}
  \|\ext\phi_k\|_{W^{2,2}(\bb R^d)}^{\frac{1+\alpha}{2}}
  \|\ext\phi_k\|_{L^{p_2}(\bb R^d)}^{\frac{1-\alpha}{2}}\\
  &=
  \|\phi_k\|_{W^{2,2}(D)}^{\frac{1+\alpha}{2}}
  \|\phi_k\|_{L^{p_2}(D)}^{\frac{1-\alpha}{2}}\\
  &\lesssim_{D,\alpha,d}
  \| \Delta \phi_k\|_{L^2(D)}^{\frac{1+\alpha}{2}}
  \|\phi_k\|_{L^{p_2}(D)}^{\frac{1-\alpha}{2}}\\
  &=
  \lambda_k^{\frac{1+\alpha}{2}}
  \|\phi_k\|_{L^{p_2}(D)}^{\frac{1-\alpha}{2}},
\end{align*}
where in the last inequality we estimated the $W^{2,2}(D)$-norm of $\phi_k$  by the $L^2(D)$-norm of $\Delta\phi_k = -\lambda_k \phi_k$ and that $||\phi_k||_{L^2(D)}=1$.
Now, we use \cite[Theorem 1]{grieser2002uniform}, which states that 
\begin{equation}
    \|\phi_k\|_{L^{\infty}(D)}
    \lesssim_{D}
    \lambda_k^{\frac{d-1}{4}}.
    \label{eq:UnifPhi}
\end{equation}
Using the interpolation version of H\"older inequality, the bound \eqref{eq:UnifPhi}, and the fact that $\|\phi_k\|_{L^2(D)}=1$, we have that 
\begin{equation}
    \|\phi_k\|_{L^{p_2}(D)}
    \lesssim_{D}
    \lambda_k^{\frac{(d-1)(1-2/p_2)}{4}}.
    \label{eq:Lp2Phi}
\end{equation}
The proof is completed by choosing $\delta$ small enough so that $\frac{(1-\alpha)(d-1)(1-2/p_2)}{8}\le \kappa$.
\end{proof}

\noindent
Consider $G^\varepsilon_D$,  the fundamental solution of $\naeps$, that is 
\begin{equation} \label{eq:defGDeps}
    \begin{cases}
	    (-\naeps G^\varepsilon_D(x,\cdot) )(y) 
	    = \delta(x-y) & \text{ if } x,y \in  D\\
		\;\;\;\;\; \;\;\;\; \;\,\,
	    \phantom{\naeps}G^\varepsilon_D(x,\cdot) = 0 & \text{ in } \partial D.
    \end{cases}
\end{equation}
Classical results of partial differential equations provide some useful properties for $G^\varepsilon$, for instance the symmetry of Green's function, that is $G_D(x,y)=G_D(y,x)$, see \cite[Theorem~1.3]{gruter1982green} for the case $d \ge3 $ and \cite[Section~6]{dolzmann1995estimates} for the case $d=2$.
 
The following lemma provides a representation for $ G^\varepsilon_D$.
It introduces a natural way of comparing it to the Green's function of $\ahom \Delta$, and gives the main intuition behind our results.
Remark that this lemma justifies the representation \eqref{eq:RepXieps} of the non-homogeneous bi-Laplacian fields.
\begin{lemma} \label{lem:repGDeps}
  Let ${\bf a} \in \Omega(\Lambda)$, and $\varepsilon \in (0,1]$. 
  Then, we have that the following representation 
  \begin{equation} \label{eq:RepGDeps}
    	G_D^\varepsilon(x,y)
	= \sum_{k = 1}^{\infty} 
	\frac{\phi_k(x) \varphi^\varepsilon_k(y)}{\ahom\lambda_k}
	= \sum_{k= 1}^{\infty} 
	\frac{\varphi^\varepsilon_k(x) \phi_k(y)}{\ahom\lambda_k},
  \end{equation}
  where $\ahom >0$ is the same constant defined in \eqref{def:a}.
 For Lebesgue almost all fixed $x \in D$, for each $\beta > \frac{d}{4}-\frac{1}{2}$  the series above converges ${\bf P}$-a.s. in the space $H^{-\beta}_0(D)$.
\end{lemma}

\begin{proof}
Remember that for any ${\bf a }\in \Omega(\Lambda)$, the Green's function is point-wise well-defined on $D^\prime:=D\times D \setminus \{ (x,x):x \in D \}$.
Call for simplicity $G^{\varepsilon}_x:=G^{\varepsilon}_D(x,\cdot)$ for some fixed $x\in D$.
We have that for $d \ge 3$, $G^\varepsilon_x \in W^{1,1}_0(D)$ (see \cite{gruter1982green}) and for $d=2$ we have that $G^\varepsilon_x \in L^1(D)$ (see \cite{taylor2013green}).
In either case, we can see $G^\varepsilon_x$ as a tempered distribution and since $\phi_k \in C^\infty_c(D)$, we get that $ \< G^\varepsilon_x, \phi_k \>$ is well-defined.

In order to calculate $\hat{G}^\varepsilon_x(k):=\<G^\varepsilon_x,\phi_k\>$, note that as $G^\varepsilon$ is the Green's function of $\naeps$ and the operator $\naeps$ is self-adjoint.
Therefore we have that 
\begin{align*}
  \<G^\varepsilon_x,\phi_k\>
  &=
  \frac{1}{\ahom \lambda_k} 
  \<G^\varepsilon_x,\naeps\varphi^\varepsilon_k\>
  \\ &=
  \frac{\varphi^\varepsilon_k(x)}{\ahom \lambda_k}.
\end{align*}
Hence,
%
\begin{align*}
  \|G^\varepsilon_x\|_{H^{-\beta}_0(D)}^2
  &=
  \sum_{k = 1}^{\infty} |\hat{G}^\varepsilon_x(k)|^2 \lambda_k^{-2\beta}
  =
  {\ahom}^{-2}
  \sum_{k= 1}^{\infty} 
  \| \varphi^\varepsilon_k(x)  \|^2
  \lambda_k^{-2\beta-2}.
\end{align*}
Finally, by integrating in $x$, we have that 
\begin{align*}
  \Big\|\|G^\varepsilon_x\|_{H^{-\beta}_0(D)}\Big\|_{L^2(D)}^2
  & \le
  \sum_{k= 1}^{\infty} 
  \| \varphi^\varepsilon_k \|^2_{L^2(D)}
  \lambda_k^{-2\beta-2}
  \\ &\le
  \sum_{k = 1}^{\infty} 
  (
    \| \phi^\varepsilon_k \|^2_{L^2(D)}
    +
    \| \varphi^\varepsilon_k-\phi^\varepsilon_k \|^2_{L^2(D)}
  )
  \lambda_k^{-2\beta-2}
  \\ &\le
  \sum_{k=1}^{\infty} 
  (
    1
    +
    \mc X_\varepsilon \lambda_k^{1+\alpha+2\kappa}
  )
  \lambda_k^{-2\beta-2}
  <\infty,\quad {\bf P}-a.s,
\end{align*}
as long as $\beta > \frac{d}{4}-\frac{1}{2}$, using that $\alpha$ and $\kappa$ can be made arbitrarily small and Weyl's law.
This also implies that $\bf P$-\textit{a.s} and almost all $x$ according to the Lebesgue measure, $\|G^\varepsilon_x\|_{H^{-\beta}_0(D)}$ is finite.
\end{proof}

\subsection*{Stochastic homogenization of the GFF}

\noindent
For this proof we will use a truncation argument and apply the following theorem.

\begin{theorem}[Theorem~3.2 in \cite{billingsley2013convergence}]
	\label{thm:Billingsley}
	Let $S$ be a metric space with metric $\rho$. Suppose that 
	$(X_{n, K}, X_{n})$ are elements of $S \times S$. If
	\begin{equation} \label{eq:thmBill}
		\lim _{K \to+\infty} \limsup _{n \to+\infty} 
		\mu(\rho(X_{n, K}, X_{n}) \geq \tau)=0
	\end{equation}
	for all $\tau>0,$ and $X_{n, K} \stackrel{d}{\longrightarrow}_{n} Z_{K} 
	\stackrel{d}{\longrightarrow}_{K} X$,  then 
	$X_{n} \stackrel{d}{\longrightarrow}_{n} X$.
\end{theorem}

We will apply the theorem by choosing $S=H^{-\beta}_0(D)$, $X = \Xi^{g}_D$, $X_n=\Xi_D^{g,\varepsilon}$ and $\mu={\bf P} \otimes {\bb P}$.
 Define the following truncations $X_{n,K}:=\Xi^{g,\varepsilon}_K$ and $Z_K:= \Xi^{g}_K$, centered Gaussian fields such that 
\begin{equation*}\label{eq:defhepsK}
  \bb E
  [
      \< \Xi^{g,\varepsilon}_K,\phi_k\>
      \<\Xi^{g,\varepsilon}_K,\phi_{k^\prime}\>
  ]
  :=
      \frac{
	\<\phi_{k},
	\varphi^{\varepsilon}_{k^\prime}\>
      }{\ahom \lambda_k}
      \1_{  \{k,k^\prime \le K \}}
\end{equation*}
and
\begin{equation*}\label{eq:defhK}
  \bb E
  [
      \<\Xi^g_K,\phi_k\>
      \<\Xi^g_K,\phi_{k^\prime}\>
  ]
  :=
  \frac{\delta_{k,k^\prime}}{\ahom \lambda_k}
  \1_{  \{ k,k^\prime \le K \}}.
\end{equation*}
By truncating the covariance of $Z_K$ and $X_{n,K}$ we essentially reduce the problem to convergence of finite dimensional Gaussian vectors.
Since $\Xi^{g,\varepsilon}_K$ and $\Xi^g_K$ are determined by their covariance structure, the convergence in distribution result will follow from  proving that their covariance matrices converge.
This will follow from Lemma~\ref{lem:growthError} since we truncated for the $k$'s to be in the set $\{1,\dots,K\}$.

In the remainder we demonstrate that \eqref{eq:thmBill} holds.
Define the error field by 
\begin{equation}\label{eq:error-GFF}
	\Xi_{\err}^{g,\varepsilon}:=
	\Xi_D^{g,\varepsilon}-
	\Xi^{g,\varepsilon}_K
	=
	\sum^{\infty}_{k=K+1} \<\Xi^{g,\varepsilon}_D, \phi_k\> \phi_k.
\end{equation}
Using the definition of the norm $H^{-\beta}_0(D)$ and the monotone convergence theorem we have
\begin{align*}
	  {\bb E}  \left[\|\Xi_{\err}^{g,\varepsilon}\|_{H^{-\beta}_0(D)}^2  \right]
&=
	{\bb  E}\left[\sum^{\infty}_{k=K+1}   
	  |\<\Xi_{\err}^{g,\varepsilon}, \phi_k\>|^2  \lambda_{k}^{-2\beta}\right]
=
	\sum^{\infty}_{k=K+1}   
	  {\bb  E}\left[|\<\Xi_{\err}^{g,\varepsilon}, \phi_k\>|^2\right]
	  \lambda_{k}^{-2\beta}.
\end{align*} 
Using the triangular inequality  first and then Lemma~\ref{lem:growthError}, we get that the sum above is bounded above (up to a multiplicative constant) by
\begin{align*}
	\sum^{\infty}_{k=K+1} 
	  \< \phi_k,
	  \varphi^{\varepsilon}_k\>
	  \lambda_{k}^{-2\beta-1}
&\le
	\sum^{\infty}_{k=K+1} 
	  \left( 1 + \mc X_{\varepsilon}\lambda_k^{ \frac{1+\alpha}{2}+\kappa}\right)^2 
	  \lambda_{k}^{-2\beta-1}
\\&\lesssim
	\sum^{\infty}_{k=K+1} 
	  \lambda_{k}^{-2\beta-1}
  +
	  \mc X_{\varepsilon}^2
	\sum^{\infty}_{k=K+1} 
	  \lambda_{k}^{-2\beta+\alpha+\kappa},
\end{align*}
which is finite as long as $\beta > \frac{d}{4}$, as we can choose $\kappa,\alpha$ to be arbitrarily small together with Weyl's law.
We can then use Markov's inequality with the product measure $\bf P \otimes {\bb P}$, to get that 
\begin{align*}
  \lim_{K \to \infty} \limsup_{n \to \infty} 
	{\bf P} \otimes {\bb P}
	\left[\| \Xi_{\err}^{g, \varepsilon}\|_{H^{-\beta}_0(D)}^2 \ge \tau \right]
	=0.
\end{align*}
Applying Theorem~\ref{thm:Billingsley}, we conclude the proof.

\bigskip

\subsection*{Stochastic homogenization of the bi-Laplacian field} 
\noindent
Let $\beta > \frac{d}{4}-\frac{1}{2}$, $\tau>0$, by Markov's inequality 
\begin{equation} \label{eq:MarkIn}
	{\bb P}  
	\left[ \|\Xi_D^{b,\varepsilon}-\ahom^{-1}\Xi_D^{b}\|_{H^{-\beta}_0(D)} \ge \tau\right]
	\le
	\frac{
	{\bb E}
	 \left[\|\Xi^{b,\varepsilon}_D-\ahom^{-1}\Xi^b_D\|^2_{H^{-\beta}_0(D)}
	 \right]}{\tau^2}.	
\end{equation}
Let $\alpha,\kappa >0$ be small enough so that $\beta - \frac{\alpha+\kappa}{2} >\frac{d}{4}-\frac{1}{2}$.
From the representations \eqref{eq:RepBiLap}, \eqref{eq:RepXieps}, and Lemma~\ref{lem:growthError} we conclude that 
\begin{align}\label{eq:ConvBilap-Control-1} 
   {\bb E} \left[\|\Xi_D^{b,\varepsilon}-\ahom^{-1}\Xi^b_D\|^2_{H^{-\beta}_0(D)} \right]
& = \nonumber
    \frac{1}{\ahom^2}
    \sum_{k\ge 1}  {\bb E}  \left[|\hat{\xi}^\varepsilon(k)-\hat{\xi}(k)|^2  \right]
    \lambda_k^{-2\beta-2}
\\ & = \nonumber
    \frac{1}{\ahom^2}
    \sum_{k\ge 1}  \|\varphi^\varepsilon_k - \phi_k\|_{L^2(D)}^2
    \lambda_k^{-2\beta-2}
\\& \lesssim \nonumber
    \sum_{k\ge 1}  \mc X_\varepsilon ^2 
    \lambda_k^{-2\beta-1+\alpha+\kappa}
\\ & \lesssim 
    \mc X_\varepsilon ^2 
\end{align}
where in the last inequality, we used Weyl's law.
Again by using Markov's inequality, we get 
\begin{equation}\label{eq:ConvBilap-Control-2} 
	{\bf P} \otimes{\bb P}
	\Big[ \|\Xi_D^{b,\varepsilon}-\ahom^{-1}\Xi^b_D\|_{H^{-2\beta}_0(D)} \ge \tau\Big]
\lesssim  
	\frac{{\bf  E} [\mc X_\varepsilon^2]}{\tau^2}
\lesssim  
	\begin{cases}
	\frac{(\varepsilon^{\alpha}|\log(\varepsilon)|)^2}{\tau^2}, & \text{ if } d = 2\\
	\frac{\varepsilon^{2\alpha}}{\tau^2}, & \text{ if } d\geq 3
\end{cases}
\end{equation}
taking $\varepsilon \to 0$ concludes the proof.

\subsection{Proof of Theorem \ref{thm:ConvDisc}}
\label{subsec-proof-bilap-discrete}

The strategy to prove the convergence result on the discrete torus will follow similar ideas as in the continuum.
Indeed, we can expand the Green's function for $\naN$ in terms of the eigenfunctions of the discrete Laplacian.
The advantage, in this case, is that the eigenfunctions are precisely the Fourier basis of $\ell^2(\bb T^d_N)$.

Discrete Fourier analysis was also the main tool of previous works on scaling limits of odometer fields, see \cite{chiarini2021constructing,cipriani2019scaling}.
 In the case of random environments note that the Fourier basis ceases to be a basis of eigenfunctions for the operator $\naN$.
Again, the key idea is to show that, for large $N$, the Fourier basis will be close to the basis of eigenvectors.

Let $g_N \in \ell^2(\bb T^d_N)$ with $\sum_{\hat{x}\in \bb T^d_N} g_N(\hat{x})=0$. 
Consider $f_N$ the solution of the equation
  \begin{equation}\label{eq:DiscPDEHom}
    \begin{cases}
      -\naN f_N (\hat{x})
      =
      g_N(\hat{x}), &
     \hat{x}\in \mathbb{T}^d_N \\
	  \,\,\; \; \; \; \; \; \; \;
      \sum_{\hat{x} \in \bb T^d_N} f_N(\hat{x}) = 0.
    \end{cases}
  \end{equation}
The next theorem states good bounds for the difference between $f_N$ and $\tilde{f}_N$, the function in  $\ell^2(\bb T^d_N)$ solving the finite difference equation
  \begin{equation}\label{eq:DiscPDEHom2}
    \begin{cases}
      -\ahom\Delta_N \tilde{f}_N (\hat{x})
\,
      =
      g_N(\hat{x}), &
     \hat{x}\in \mathbb{T}^d_N \\
	  \,
      \sum_{\hat{x} \in \bb T^d_N} f_N(\hat{x}) = 0,
    \end{cases}
  \end{equation}
where $\ahom$ is a positive constant that only depends on the law of ${\bf a}$.
The next theorem is going to serve as the equivalent of Theorem~\ref{thm:AKM} for the discrete context and is a simple adaptation of \cite[Corollary 1.2]{gloria2014optimal}.

\begin{theorem}[Corollary 1.2, \cite{gloria2014optimal}] \label{thm:GON}
  There exists a deterministic positive constant $ \ahom$ only depending on
  the law of ${\bf a}$  and $d$ with the following property.  Given $N \ge1$ with $f
  \in C^\infty (\bb T^d)$, let $f_N$ be the solution of
  \eqref{eq:DiscPDEHom} and $\tilde{f}_N$ be the solution of
  \eqref{eq:DiscPDEHom2}. We have
  \begin{equation}\label{eq:thmGON}
    {\bf E} \left[\|f_N-\tilde{f}_N\|^2_{\ell^2(\bb T^d_N)}\right]
    \lesssim
    c_d(N) N^{-2}  \|g_N\|^2_{\ell^2(\bb T^d_N)},
  \end{equation}
  where $c_2(N)=\log(N)$  and equal to $c_d(N)=1$ for $d \ge 3$.
\end{theorem}
We will be interested in the case $g_N=-\ahom\lambda^{(N)}_k \phi^N_k$, where
\[
  \lambda^{(N)}_k
  := 4  N^2 \sum_{i=1}^d
  \sin^2\left(\frac{\pi k_i}{N} \right),
\]
are the eigenvalues of the normalized discrete Laplacian operator.
Very conveniently, in this case, $\tilde{f}_N=\phi^N_k$.
In the following lemma, we will denote by $\varphi^N_k$ the solution to \eqref{eq:DiscPDEHom}.

Now,  let us state the representation of the Green's functions $G^{N,{\bf a}}$
defined in \eqref{def-Green-function}. 

\begin{lemma}\label{lem:exp-for-G}
  For all ${\bf a} \in \Omega(\Lambda)$ and any $N \ge 1$, we have
  \begin{equation}
    G^{N,{\bf a}}( \hat{x}, \hat{y})
    =
    G^{N,{\bf a}}(\hat{y}, \hat{x})
    =
    \frac{1}{2d}\frac{1}{N^{d}}
    \sum_{k \in \bb Z^d_N \setminus\{0\} }
   \frac{\phi_k(\hat{x}) \overline{\varphi^N_{k}(\hat{y})}}{\ahom\lambda^{(N)}_k},
  \end{equation}
  for all $\hat{x}, \hat{y}\in \bb T^d_N$.
\end{lemma}
This proof is similar to the proof of Lemma~\ref{lem:repGDeps} and will be therefore omitted.
It is much simpler as the sum only has a finite number of terms.
The next lemma is just an application of Theorem~\ref{thm:GON}.

\begin{lemma} \label{lem:bound-homog}
  We have that
  \begin{equation}\label{eq:bound-homog}
     {\bf E} [\| \varphi^N_k - \phi_k\|^2_{ \ell^2(\bb T^d_N)}]
     \lesssim_{d,\Lambda}
      c_d(N)
     \frac{\|k\|^4}{N^2},
  \end{equation}
  where $c_2(N):= \log(N)$ for $d=2$ and $c_d(N):=1$ for $d \ge 3$. 
\end{lemma}
\subsection*{Stochastic homogenization of the discrete non-homogeneous GFF}

\noindent
The proof is very similar to the proof of Theorem~\ref{thm:ConvCont} (1) hence we will point out the main differences.
We use Theorem~\ref{thm:Billingsley} with $S=H^{-\beta}_0(\bb T^d)$, with $\beta> d/4$ and $X:=\Xi^g_D$, $X_{N,K}:=\Xi_{N,K}^{g,{\bf a}}$ and $Z_K := \ahom^{-1/2}\Xi^g_K$, where
\begin{equation*} 
  \Xi_{N,K}^{g,{\bf a}} 
  :=
  \sum_{k \in  \bb Z^d_{N\wedge K}\setminus \{0\} }
  \< \Xi^{g,{\bf a}}_{D,N},\phi_k \>\phi_k,
\end{equation*}
and $a \wedge b := \min\{a,b\}$, $\Xi^{g,{\bf a}}_{D,N}$ was defined in
\eqref{def:formalField} and
\begin{equation}
  \Xi_{D,K}^{g} 
  := 
  \sum_{k \in  \bb Z^d_{ K}\setminus \{0\} }   \< \Xi^g_D,\phi_k \>\phi_k.
\end{equation}
The proof of the  argument follows similarly but using Lemma~\ref{lem:bound-homog} instead of Lemma~\ref{lem:growthError} to prove the convergence of $\Xi_{N,K}^{g,{\bf a}}$ to $\ahom^{-1/2}\Xi^{g}_{D,K}$.

\bigskip
\subsection*{Stochastic homogenization of the discrete non-homogeneous bi-Laplacian field}

\noindent
Let $\tau>0$, we want to prove that
\[
  \lim_{N\rightarrow \infty} { \bf P} \otimes {\bb P} \left[ \|\Xi_{D,N}^{b, {\bf
  a}}-\ahom^{-1}\Xi^b_D\|_{H^{-\beta}_0(D)}> \tau\right] =0.
\]
Note the trivial identity
\[
  \Xi_{D,N}^{b, {\bf a}}-\ahom^{-1}\Xi^b_D = \left ( \Xi^{b, {\bf a}}_{D,N} -
  \ahom^{-1} \Xi^b_{D,N} \right ) + \ahom^{-1} (\Xi^b_{D,N}-\Xi^b_D).
\]
For the discrete case, we compare the non-homogeneous field to a discrete homogeneous one.
 That is, we will first show that $\Xi^{b}_{D, N}$ (the homogeneous formal field) converges to $\Xi_D^{b}$ in probability according to an appropriate Sobolev norm,  as long as we choose a suitable coupling between $\xi$ and $\xi_N$ in \eqref{eq:defEtaN}.
 For this, we will take $\xi_N(\hat{x})=N^{d/2}\xi(\1_{B_N(\hat{x})})$, where $\xi$ is the same sample of the noise used in the definition of \eqref{eq:defBiLap}, $\1_A$ denotes the indicator function of the set $A$ and $B_N(\hat{x}):=\hat{x}+[-\frac{1}{2N},-\frac{1}{2N}]^d$.
 Secondly, we prove that $\Xi_{\err}^{b, N}:=\Xi_{D,N}^{b, {\bf a}}-\ahom^{-1}\Xi^{b}_{D}$ vanishes in this same Sobolev space using estimates given in \cite{gloria2014optimal}.

We will prove the first point described above in the next proposition.

\begin{proposition}\label{prop:conv-homog-disc}
  For all $\beta>\frac{d}{4}-1$, $D=\mathbb{T}^d$ we have that   
   \begin{equation}\label{eq:conv-homog-disc}
      {\bb E} \left[\|\Xi^{b}_{D,N}- \ahom^{-1}\Xi^{b}_D\|^2_{H^{-\beta}_0(D)}	\right]
  \lesssim
      N^{d-4-4\beta}.
   \end{equation}

\end{proposition}
\begin{proof}

Again, it will be convenient to study the action of the field on the basis of eigenfunctions of the Laplacian.
Remember that, as we are on the torus (both in the discrete and the continuous cases), such basis of eigenfunctions is given by the Fourier basis $\phi_{k}:= \exp( 2\pi \iota k\cdot x )$.
Let us start by calculating explicitly $\Xi^b_{N}(\phi_k)$ for $k \in \bb Z^d_N$, 
\begin{align} \label{eq:XiNhomk} 
   \Xi^b_{N}(\phi_k)
&=
\frac{1}{2d} \frac{1}{N^{d/2}} \nonumber
    \sum_{\hat{z} \in \bb T^d_N}  \sum_{\hat{y} \in \bb T^d_N}
    G^N(\hat{z},\hat{y}) \xi ( N^{d/2}\1_{B_N(\hat{y})}) 
    \overline{\phi_k(\hat{z})}
\\ &= \nonumber
   \frac{1}{2d} \frac{1}{\lambda_k^{(N)}}
    \xi \left(
      \sum_{\hat{y} \in \bb T^d_N} \phi_k(\hat{y})\1_{B_N(\hat{y})}
    \right) 
\\ &=:
 (2d\lambda_k^{(N)})^{-1}
    \xi (\tilde{\phi}^N_k)
\end{align}
where in the first identity we used that $\sum_{\hat{z} \in \mathbb{T}^d_N} G^N(\hat{z},\hat{y})=0$ for any $\hat{y}$.
A simple computation shows that (see \cite[Lemma~7]{cipriani2018scaling})
\begin{equation}\label{eq:art-disc-1}
  \frac{1}{\lambda^{(N)}_k} = \frac{1}{\lambda_k} + \mc O( N^{-2}).
\end{equation}

On the other hand, using the expansion of the Green's function $G_{\bb T^d}$ on the torus, we get that $ \Xi^b_D(\phi_k)= {\lambda_k}^{-1} \xi(\phi_k)$.
Notice that we can estimate $\|\phi_k - \tilde{\phi}_k^N\|^2_{L^2(\bb T^d)} \lesssim \frac{\|k\|^2}{ N^2}$.
Therefore, for $k \in \bb Z^d_N\setminus\{0\}$, we have 
\begin{align*}
  {\bb E}\left[| \<\Xi^b_{D,N} - \ahom^{-1}\Xi_D^{b},\phi_k\>|^2\right]
  \lesssim
  \frac{1}{\|k\|^{2}N^2}.
\end{align*}
As $\Xi^{b}_{D,N}(\phi_k)=0$ if $k \not \in \bb Z^d_N$,  we get that for any $\beta$,
we have
\begin{align*}
  {\bb E} \left[\|\Xi^b_{D,N}- \ahom^{-1}\Xi_D^{b}\|^2_{H^{-\beta}_0(D)}	\right]
  &\lesssim_{\beta,d}
  \sum_{k \in \bb Z^d_N \setminus \{0\}} 
  {\bb E}[| \Xi^b_N(\phi_k) - \ahom^{-1}\Xi_D^{b}(\phi_k)|^2] \|k\|^{-4\beta}
  \\ &\phantom{\lesssim_{\beta,d}} +
 \ahom^{-2} \sum_{k \in \bb Z^d \setminus \bb Z^d_N} 
  {\bb E} [| \Xi^b_D(\phi_k )|^2  ]\|k\|^{-4\beta}
  \\&\lesssim_{\beta,d,\ahom}
    \frac{N^{d-2-4\beta}}{N^2}
  + 
  N^{d-4-4\beta},
\end{align*}
proving \eqref{eq:conv-homog-disc}. 
\end{proof}

We now proceed to show that $\Xi_{\err}^{b, N}$ converges in probability to $0$ in the space $H^{-\beta}_0(D)$  for $\beta>\frac{d}{4}-\frac{1}{2}$.

\begin{proposition}\label{prop:bound-norm}
For $d \ge 2$,  as $N \longrightarrow \infty$, we have
\[
  {\bf E} \otimes {\bb E} \Bigg[
  \|\Xi_{\err}^{b, N}\|^2_{H^{-\beta}_0(D)} \Bigg] 
  \lesssim_{\beta,d,\Lambda}c_d(N) N^{d-2-4\beta},
\]
where $c_d(N)=\log (N)$ for $d=2$ and $c_d(N)=1$ otherwise.
\end{proposition}

\begin{proof}
Define a Gaussian vector $(\Xi_{\err}^{b, N}(\hat{x}))_{\hat{x} \in \bb T^d_N}$ 
by
\begin{equation}\label{eq:defbareta}
  \Xi_{\err}^{b, N}(\hat{x})
:=
  \sum_{\hat{y} \in \bb T^d_N} (G^N(\hat{x},\hat{y})-\ahom^{-1}G^{N, {\bf
  a}} (\hat{x},\hat{y})) \xi_N(\hat{x}).
\end{equation}
Call
\[
  \Xi_{\err}^{b, N} = \frac{1}{2d} \frac{1}{N^{d/2}} \sum_{\hat{y} \in \bb T^d_N}  \Xi_{\err}^{b, N}(\hat{y})\delta_{\hat{y}}
\]
and remark that here we are committing the same abuse of notation as described in Remark~\ref{rem:abuse}.

Let us start by computing $\Xi_{\err}^{b,N} (\phi_k)$, for fixed $k \in \bb Z^d$.
\begin{align*}
  {\bb E}  [|\Xi_{\err}^{b,N} (\phi_k) |^2  ] = \frac{1}{(2d)^2}\frac{1}{N^d}
  \sum_{\hat{x},\hat{y} \in \bb T^d_N}
   \phi_k (\hat{x}-\hat{y})
   {\bb E} [\Xi_{\err}^{b, N}(\hat{x})\Xi_{\err}^{b, N}(\hat{y})].
\end{align*}

Expanding the previous expression we get
\begin{align*}
  &
  \sum_{\hat{x},\hat{y} \in \bb T^d_N}
  \phi_k (\hat{x}-\hat{y})
  {\bb E} [\Xi_{\err}^{b, N}(\hat{x})\Xi_{\err}^{b, N}(\hat{y})]
  \\& =
  \frac{1}{N^{2d}}
  \sum_{\hat{x},\hat{y} \in \bb T^d_N}
  \sum_{k_1,k_2 \in \bb Z^d_N \setminus\{0\}}
  \phi_{k-k_1} (\hat{x}) \overline{\phi_{k-k_2} (\hat{y}) }
  \sum_{\hat{z} \in \bb T^d_N}
  \frac{\phi_{k_1}(\hat{z})-\varphi^N_{k_1}(\hat{z})}{
  \ahom\lambda^{(N)}_{k_1}}
  \overline{
  \frac{\phi_{k_2}(\hat{z})-\varphi^N_{k_2}(\hat{z})}{
  \ahom\lambda^{(N)}_{-k_2}}}
  \\& =
  \sum_{k_1,k_2 \in \bb Z^d_N \setminus\{0\}}
  \delta_{k,k_1} \delta_{k,k_2}
  \sum_{\hat{z} \in \bb T^d_N}
  \frac{\phi_{k_1}(\hat{z})-\varphi^N_{k_1}(\hat{z})}{
  \ahom\lambda^{(N)}_{k_1}}
  \overline{
  \frac{\phi_{k_2}(\hat{z})-\varphi^N_{k_2}(\hat{z})}{
  \ahom\lambda^{(N)}_{-k_2}}}
\\& =
  \sum_{\hat{z} \in \bb T^d_N}
  \frac{|\phi_{k}(\hat{z})-\varphi^N_{k}(\hat{z})|^2}{
  ( \ahom\lambda^{(N)}_{k} )^2}.
\end{align*}
By using Lemma \ref{lem:bound-homog}, we get
 \[
  {\bf E} \otimes {\bb E} \Bigg[
  \|\Xi_{\err}^{b, N}\|^2_{H^{-\beta}_0(D)} \Bigg] =
    {\bf E} \Bigg[ \frac{1}{(2d)^2N^{d}}
    \sum_{\hat{x},\hat{y} \in \bb T^d_N}
    \phi_k (\hat{x}-\hat{y})
    {\bb E} [\Xi_{\err}^{b, N}(\hat{x})\Xi_{\err}^{b, N}(\hat{y})] \Bigg]
    \lesssim
    \frac{c_d(N)}{N^2}.
 \]

\end{proof}

Finally, let $\beta > \frac{d}{4} - \frac{1}{2}$ and choose $\kappa>0$ sufficiently small so that
$\beta> \frac{d}{4}- \frac{1}{2} + \frac{\kappa}{2}$, then we have
\begin{align*}
 { \bf P} \otimes {\bb P} \left[ \|\Xi_{D,N}^{b, {\bf
 a}}-\ahom^{-1}\Xi^b_D\|_{H^{-\beta}_0(D)}> N^{-\kappa}\right]
&  \le
  {\bf P} \otimes {\bb P} \left[ \|\Xi_{D,N}^{b, {\bf
  a}}-\ahom^{-1}\Xi^b_D\|_{H^{-\beta}_0(D)}> \frac{N^{-\kappa}}{2}\right]
  \\ & \phantom{m} +
  {\bf P} \otimes {\bb P} \left[ \|\Xi_{D,N}^{b}-\Xi^b_D\|_{H^{-\beta}_0(D)}>
  \frac{N^{-\kappa}}{2\ahom}\right]
\\ &  \lesssim 
  c_d(N)N^{d-2-4\beta+2\kappa}
  + 
  N^{d-4-4\beta+2\kappa},
\end{align*}
which vanishes as $N$ goes to infinity.

\section{Discussion}\label{sec:conclusion}

In this section, we provide a few remarks regarding possible generalizations and comparison to other results.

\subsection*{Other results on large scale behaviour of Gaussian fields in random environments}

In this article, we focused on random environments that took the form of random uniformly elliptic differential operators.
Results regarding the maxima of such fields inside of a percolation cluster in dimension $2$ were obtained in \cite{schweiger2022maximum}. 
On the other hand, the thermodynamic limit and the height fluctuation of (possibly Gaussian) random fields associated to Gibbs measures with random masses are studied in \cite{dario2021convergence,dario2021random}. 

\subsection*{Convergence in probability of the non-homogeneous GFF}

A natural question is whether the convergence of the (non-homogeneous) GFF can be improved to a convergence in probability.
For this, we need a coupling in terms of a elliptic PDE which allows us to employ stochastic homogenization techniques.
Indeed, there is a natural coupling by writing the desired fields as the solution to PDEs similar to the one found in \cite{gu2017generalized}, in which the authors define the notion of \emph{generalized Gaussian free field}.
  
Unfortunately, to be able to extract results from such coupling, we would need to obtain a bound similar to the one found in Lemma~\ref{lem:growthError} but for the quantity $\| b_{\varepsilon}\nabla \varphi^\varepsilon_k - \bar{b} \nabla\phi_k\|_{L^2(D)}$ where $b_\varepsilon := \sqrt{{\bf a}_{\varepsilon}}$ and $\bar{b}:= \sqrt{\bar{{\bf a}}}$.
This is not expected to converge as $\varepsilon$ vanishes.
Indeed, one can see in \cite{Armstrong2019,gloria2014optimal} that $\varphi^\varepsilon_k$ needs the so called \emph{first-order correctors} in order to converge to its homogenized counter part in $H^{1}_0(D)$.

\subsection*{Discretized domains} 
In this article, we derive scaling limits in the continuous setting in a domain and for the discrete setting on a torus.
However, we do believe that such results could be extended to  discretized domains.
The proof would require both adapting the techniques of quantitative discrete stochastic homogenization near the boundary.
One would also need to account for the fact that the discretized eigenfunctions of the Laplacian cease to   be the same as the eigenfunctions of the discretized Laplacian.
One should be able to bound such deterministic errors in such context, similarly to \cite{burago2015graph} where they bound such error for manifolds without boundary.

\subsection*{General fractional fields} 
We focused in the case of the GFF and bi-Laplacian fields, one could wonder what happens for general fractional Gaussian fields, see \cite{Lodhia2016}.
This seem much harder, as one would need to deal with stochastic homogenization for non-local operators, which seems far from the scope of the current available theory of quantitative stochastic homogenization.

\subsection*{Restricted bi-Laplacian field}
In this article, we used the \emph{eigenvalue fractional field} definition for the bi-Laplacian field on a domain $D$.
This was particularly important in order to use stochastic homogenization techniques to couple both the homogenized and non-homogeneous versions of the bi-Laplacian field.

One can however wonder if it is possible to use the so-called \emph{zero-boundary} definition of fractional Gaussian fields which is more similar in spirit to our definition of the GFF, i.e, based on covariance functions.
However, in this setting, the covariance of the bi-Laplacian field is the Green's function of the bi-Laplacian operator with Neumann instead of Navier boundary conditions, see Remark~\ref{rem:BoundaryCondGreen}.

We believe a notion of non-homogeneous bi-Laplacian field would also converge in law (instead of in probability) to a homogeneous one via stochastic homogenization techniques.
To do so, one would just need to prove the results we used from \cite{Armstrong2019,gloria2014optimal} in the context of differential operators of order $4$.

\subsection*{Acknowledgements}
The authors would like to thank Jean-Christophe Mourrat, Alessandra Cipriani and Rajat Hazra for fruitful conversations and for providing useful references for this article.

\bibliographystyle{alpha}
\bibliography{library}

\newcommand{\etalchar}[1]{$^{#1}$}
\begin{thebibliography}{LMPU16}

\bibitem[AKM19]{Armstrong2019}
S.~Armstrong, T.~Kuusi, and J.-C. Mourrat.
\newblock {\em {Quantitative Stochastic Homogenization and Large-Scale
  Regularity}}, volume 352 of {\em Grundlehren der mathematischen
  Wissenschaften}.
\newblock Springer International Publishing, Cham, 2019.

\bibitem[BD98]{braides}
A.~Braides and A.~Defranceschi.
\newblock {\em Homogenization of multiple integrals}, volume 112.
\newblock Oxford Lecture Series in Mathematics and its Applications. The
  Clarendon Press, Oxford University Press, New York, 1998.

\bibitem[Ber16]{Berestycki2016IntroductionTT}
N.~Berestycki.
\newblock Introduction to the gaussian free field and liouville quantum
  gravity.
\newblock {\em Lecture notes}, 2016.

\bibitem[BFFO17]{Bella2017StochasticHO}
P.~Bella, B.J. Fehrman, J.~Fischer, and F.~Otto.
\newblock Stochastic homogenization of linear elliptic equations: Higher-order
  error estimates in weak norms via second-order correctors.
\newblock {\em SIAM J. Math. Anal.}, 49:4658--4703, 2017.

\bibitem[BIK15]{burago2015graph}
D.~Burago, S.~Ivanov, and Y.~Kurylev.
\newblock A graph discretization of the laplace--beltrami operator.
\newblock {\em Journal of Spectral Theory}, 4(4):675--714, 2015.

\bibitem[Bil13]{billingsley2013convergence}
P.~Billingsley.
\newblock {\em {Convergence of probability measures}}.
\newblock John Wiley {\&} Sons, 2013.

\bibitem[BL12]{bergh2012interpolation}
J.~Bergh and J.~L{\"o}fstr{\"o}m.
\newblock {\em Interpolation spaces: an introduction}, volume 223.
\newblock Springer Science \& Business Media, 2012.

\bibitem[CdGR19]{cipriani2019scaling}
A.~Cipriani, J.~de~Graaff, and W.~M Ruszel.
\newblock Scaling limits in divisible sandpiles: a fourier multiplier approach.
\newblock {\em Journal of Theoretical Probability}, pages 1--28, 2019.

\bibitem[CHR18]{cipriani2018scaling}
A.~Cipriani, R.~S. Hazra, and W.~M Ruszel.
\newblock Scaling limit of the odometer in divisible sandpiles.
\newblock {\em Probability theory and related fields}, 172(3):829--868, 2018.

\bibitem[CJR21]{chiarini2021constructing}
L.~Chiarini, M.~Jara, and W.~M Ruszel.
\newblock Constructing fractional gaussian fields from long-range divisible
  sandpiles on the torus.
\newblock {\em Stochastic Processes and their Applications}, 140:147--182,
  2021.

\bibitem[CW84]{ap2}
R.G. Carbonell and S.~Whitaker.
\newblock Heat and mass transfer in porous media.
\newblock {\em Fundamentals of Transport Phenomena in Porous Media}, pages
  121--198, 1984.

\bibitem[Dar21]{dario2021convergence}
Paul Dario.
\newblock Convergence to the thermodynamic limit for random-field random
  surfaces.
\newblock {\em arXiv preprint arXiv:2105.03940}, 2021.

\bibitem[DHP21]{dario2021random}
Paul Dario, Matan Harel, and Ron Peled.
\newblock Random-field random surfaces.
\newblock {\em arXiv preprint arXiv:2101.02199}, 2021.

\bibitem[DM95]{dolzmann1995estimates}
G.~Dolzmann and S.~M{\"u}ller.
\newblock Estimates for green's matrices of elliptic systems by $l^p$ theory.
\newblock {\em Manuscripta mathematica}, 88(1):261--273, 1995.

\bibitem[DNPV12]{di2012hitchhikers}
Eleonora Di~Nezza, Giampiero Palatucci, and Enrico Valdinoci.
\newblock Hitchhiker's guide to the fractional sobolev spaces.
\newblock {\em Bulletin des sciences math{\'e}matiques}, 136(5):521--573, 2012.

\bibitem[Fun05]{funaki}
T.~Funaki.
\newblock {\em Stochastic interface models}.
\newblock Lectures on probability theory and statistics. Lect. Notes in Math.,
  2005.

\bibitem[GM17]{gu2017generalized}
Y.~Gu and J.-C. Mourrat.
\newblock On generalized gaussian free fields and stochastic homogenization.
\newblock {\em Electronic Journal of Probability}, 22, 2017.

\bibitem[GNO14]{gloria2014optimal}
A.~Gloria, S.~Neukamm, and F.~Otto.
\newblock An optimal quantitative two-scale expansion in stochastic
  homogenization of discrete elliptic equations.
\newblock {\em ESAIM: Mathematical Modelling and Numerical Analysis},
  48(2):325--346, 2014.

\bibitem[GO11]{gloria2011optimal}
A.~Gloria and F.~Otto.
\newblock {An optimal variance estimate in stochastic homogenization of
  discrete elliptic equations}.
\newblock {\em The Annals of Probability}, 39(3):779 -- 856, 2011.

\bibitem[Gri02]{grieser2002uniform}
D.~Grieser.
\newblock Uniform bounds for eigenfunctions of the laplacian on manifolds with
  boundary.
\newblock {\em Communications in Partial Differential Equations},
  27(7-8):1283--1299, 2002.

\bibitem[GW82]{gruter1982green}
M.~Gr{\"u}ter and K.-O. Widman.
\newblock The green function for uniformly elliptic equations.
\newblock {\em Manuscripta Mathematica}, 37(3):303--342, 1982.

\bibitem[Hai15]{Hairer}
M.~Hairer.
\newblock {Introduction to regularity structures}.
\newblock {\em Brazilian Journal of Probability and Statistics}, 29(2):175 --
  210, 2015.

\bibitem[Jur80]{jurinskii1980dirichlet}
V.V. Jurinskii.
\newblock On a dirichlet problem with random coefficients.
\newblock In {\em Stochastic Differential Systems Filtering and Control}, pages
  344--353. Springer, 1980.

\bibitem[Ken01]{Kenyon}
R.~Kenyon.
\newblock {Dominos and the Gaussian Free Field}.
\newblock {\em The Annals of Probability}, 29(3):1128 -- 1137, 2001.

\bibitem[KLO12]{komo}
T.~Komorowski, C.~Landim, and S.~Olla.
\newblock {\em Fluctuations in Markov processes}, volume 345.
\newblock Grundlehren der Mathematischen Wissenschaften, Springer, Heidelberg,
  2012.

\bibitem[Koz79]{kozlov1979averaging}
S.~M. Kozlov.
\newblock Averaging of random operators.
\newblock {\em Matematicheskii Sbornik}, 151(2):188--202, 1979.

\bibitem[Lax02]{lax2002functional}
P.~D. Lax.
\newblock Functional analysis.
\newblock {\em John Wiley\&Sons Inc. Publication}, 2002.

\bibitem[LMPU16]{Levine}
L.~Levine, M.~Murugan, Y.~Peres, and B.E. Ugurcan.
\newblock The divisible sandpile at critical density.
\newblock {\em Ann. Henri Poincar\'e}, pages 1677--1711, 2016.

\bibitem[LSSW16]{Lodhia2016}
A.~Lodhia, S.~Sheffield, X.~Sun, and S.~S. Watson.
\newblock {Fractional Gaussian fields: A survey}.
\newblock {\em Probability Surveys}, 13(0):1--56, 2016.

\bibitem[PV79]{papanicolau1979boundary}
G.~Papanicolau and S.~Varadhan.
\newblock Boundary value problems with rapidly oscillating random coefficients.
\newblock {\em Volume}, 2:835--873, 1979.

\bibitem[SGD{\etalchar{+}}16]{ap3}
P.~Sapin, A.~Gourbil, P.~Duru, F.~Fichot, M.~Prat, and M.~Quintard.
\newblock Reflooding with internal boiling of a heating model porous medium
  with mm-scale pores.
\newblock {\em International Journal of Heat and Mass Transfer}, 99:512--520,
  2016.

\bibitem[SW13]{Sun2013UniformSF}
X.~Sun and W.~Wu.
\newblock Uniform spanning forests and the bi-laplacian gaussian field.
\newblock {\em arXiv: 1312.0059v1}, 2013.

\bibitem[SZ22]{schweiger2022maximum}
Florian Schweiger and Ofer Zeitouni.
\newblock The maximum of log-correlated gaussian fields in random environments.
\newblock {\em arXiv preprint arXiv:2205.07210}, 2022.

\bibitem[Tar09]{tar}
L.~Tartar.
\newblock {\em The general theory of homogenization: a personalized
  introduction}, volume~7.
\newblock Lecture Notes of the Unione Matematica Italiana, Springer-Verlag,
  Berlin, 2009.

\bibitem[TKB13]{taylor2013green}
J.L. Taylor, S.~Kim, and R.M. Brown.
\newblock The green function for elliptic systems in two dimensions.
\newblock {\em Communications in Partial Differential Equations},
  38(9):1574--1600, 2013.

\bibitem[Vaz07]{ap1}
J.L. Vazquez.
\newblock The porous medium equation: Mathematical theory.
\newblock {\em Oxford University Press, USA}, 2007.

\end{thebibliography}

\end{document}